\documentclass[11pt]{amsart}
\usepackage[latin1]{inputenc}
\usepackage{amssymb}
\usepackage{enumerate}
\usepackage[colorlinks, bookmarks=true]{hyperref}
\usepackage{cleveref}
\usepackage{mathtools}

\usepackage[T1]{fontenc}
\usepackage{mathptmx}

\usepackage[colorlinks, bookmarks=true]{hyperref}

\newtheorem{theorem}{Theorem}[section]
\newtheorem{proposition}[theorem]{Proposition}
\newtheorem{lemma}[theorem]{Lemma}
\newtheorem{corollary}[theorem]{Corollary}
\theoremstyle{definition}
\newtheorem{remark}[theorem]{Remark}

\newtheorem{example}[theorem]{Example}

\numberwithin{equation}{section}

\newcommand{\unit}{\mathbf{1}}

\def\J#1#2#3{ \left\{ #1,#2,#3 \right\} }

\def\11{\textbf{$1$}}

\newcommand{\set}[2]{\ensuremath{\{ #1\,:\; #2\}}}
\newcommand{\nor}[1]{\ensuremath{\|#1\|}}
\newcommand{\e}{\ensuremath{\varepsilon}}

\begin{document}
	
\title[Quadratic \& additive mappings on operator commuting elements]{Quadratic \& additive mappings on operator commuting elements in JBW$^*$-algebras}

\author[G.M. Escolano]{Gerardo M. Escolano}

\address{Departamento de An{\'a}lisis Matem{\'a}tico, Facultad de Ciencias, Universidad de Granada \hyphenation{Gra-nada}, 18071 Granada, Spain.}
\email{gemares@correo.ugr.es}

\author[J. Hamhalter]{Jan Hamhalter}

\address{Department of Mathematics, Faculty of Electrical Engineering, Czech Technical University in Prague, Technicka 2, 166 27, Prague 6, Prague, Czech Republic.}
\email{hamhalte@math.feld.cvut.cz}

\author[A.M. Peralta]{Antonio M. Peralta}

\address{Departamento de An{\'a}lisis Matem{\'a}tico, Facultad de Ciencias, Universidad de Granada, 18071 Granada, Spain.
Instituto de Matem{\'a}ticas de la Universidad de Granada (IMAG).}
\email{aperalta@ugr.es}

\author[A.R. Villena]{Armando R. Villena}

\address{ Departamento de An{\'a}lisis Matem{\'a}tico, Facultad de Ciencias, Universidad de Granada, 18071 Granada, Spain.
Instituto de Matem{\'a}ticas de la Universidad de Granada (IMAG).}
\email{avillena@ugr.es}



\begin{abstract} Let $\mathfrak{A}$ and $\mathfrak{B}$ be JBW$^*$-algebras whose sets of unitaries are denoted by  $\mathcal{U}(\mathfrak{A})$ and $\mathcal{U}(\mathfrak{B})$, respectively. We show that $\mathcal{U}(\mathfrak{A})$ is closed for Jordan products of operator commuting pairs inside itself. Assuming that  $\mathfrak{A}$ and $\mathfrak{B}$ are JBW$^*$-algebras without direct summands of type $I_1$ or $I_2$, we prove that for each bicontinuous bijection $\Phi : \mathcal{U}(\mathfrak{A}) \rightarrow \mathcal{U}(\mathfrak{B})$ satisfying $\Phi (u \circ v) = \Phi (u)\circ \Phi (v),$ whenever $u$ and $v$ are operator commuting unitaries in $\mathfrak{A}$,  there exist a linear Jordan $^*$-isomorphism $\theta: \mathfrak{A} \rightarrow \mathfrak{B}$, a real linear mapping $\beta: \mathfrak{A_{sa}}\rightarrow Z(\mathfrak{B}_{sa})$, and an invertible central element $c \in \mathfrak{B}_{sa}$ such that $$ \Phi(e^{i a}) = e^{i \beta (a)}\circ e^{i  c\circ\theta(a)} = e^{i \beta(a)} \circ \theta \left( e^{i  \theta^{-1}( c )\circ a}\right),$$  for all $a\in \mathfrak{A}_{sa}$. The conclusion improves when $\mathfrak{A}$ is a JBW$^*$-algebra factor not of type $I_2$.

Assuming that $\mathfrak{A}$ has no direct summands of type $I_2$, we show that any mapping $\Phi : \mathfrak{A}_{sa} \to \mathfrak{B}_{sa}$
satisfying the following conditions:
\begin{enumerate}[(a)]
	\item $\Phi\big(U_a(b)\big) = U_{\Phi(a)}\big(\Phi(b)\big)$ whenever $a$ and $b$ operator commute in $\mathfrak{A}_{sa}$,
	\item $\Phi$ is additive on operator commuting elements,
\end{enumerate}
enjoys the following properties: First, $\Phi(\unit)$ is a tripotent in $\mathfrak{B}_{sa}$, the Peirce-2 subspace 
$\mathfrak{B}_2\!\left(\Phi(\unit)\right)$ carries a natural structure of JBW$^*$-algebra whose self-adjoint part contains the range of $\Phi$, and makes the corresponding mapping
$\Phi : \mathfrak{A}_{sa} \to \mathfrak{B}_2\!\left(\Phi(\unit)\right)_{sa}$ a unital (real linear) Jordan homomorphism. Furthermore, if $\Phi$ is bijective, then $\Phi(\unit)$ is a central symmetry in $\mathfrak{B}$, and 
$\Phi$ is an isometric real-linear Jordan isomorphism from $\mathfrak{A}_{sa}$ onto 
$\mathfrak{B}_2\!\left(\Phi(\unit)\right)_{sa}$.
\end{abstract}

\maketitle

\section{Introduction}

In quantum theory, observables are represented by self-adjoint operators on a Hilbert space $ H $, or more generally by the self-adjoint part of a von Neumann algebra. Simultaneous measurability is commonly regarded as equivalent to the commutativity of observables; nevertheless, several authors have argued that even nowhere-commuting observables can be measured simultaneously \cite{Ozawa2011}. 
The sum of two non-commuting self-adjoint operators defines a legitimate observable which satisfies uncertainty relations with each of the original ones. From this perspective, the addition of observables acquires an unambiguous physical meaning only when the involved operators commute. Multiplication of observables has an even greater difficulty, since the self-adjoint part of such a von Neumann algebra need not be, in general, closed under the underlying associative product. For this reason, these considerations have led to the concept of \emph{piecewise structures}, introduced for example by C. Heunen and M. L. Reyes in \cite{HeunenReyes2014}, where the basic idea is to formalize the relation of commutativity and consider algebraic operations only for pairs of elements lying in this relation. In this paper we shall consider the weaker variant of piecewise closed subsets of a general algebra. Suppose $\mathfrak{A}$ is a (not necessarily associative) real or complex algebra whose product is denoted by ``$a\bullet b$''. A \emph{piecewise closed} subset of $\mathfrak{A}$ is a subset $\mathcal{N}\subseteq \mathfrak{A}$ together with a commeasurability relation (i.e., a reflexive and symmetric binary relation) $\Sigma\subseteq \mathcal{N}\times \mathcal{N}$ such that $a\bullet b\in \mathcal{N}$ for all $(a,b)\in \Sigma$.\smallskip 

Examples of piecewise closed subsets of a unital C$^*$-algebra $A$ include the sets Proj$(A)$, $\mathcal{N}(A)$, $A_{sa}$, $\mathcal{U} (A)$,  $A^{+},$ and $A^{++}$, of all projections, normal elements, self-adjoint elements, unitaries, positive elements, and positive invertible elements, respectively, where commeasurability is given by commutativity.\smallskip

Let $(\mathcal{N}_1, \Sigma_1)$ and $(\mathcal{N}_2, \Sigma_2)$ be two piecewise closed subsets of a real or complex algebra $(\mathfrak{A},\bullet)$. A \emph{piecewise homomorphism} from $(\mathcal{N}_1, \Sigma_1)$ to $(\mathcal{N}_2, \Sigma_2)$ is a mapping $\Phi : \mathcal{N}_1\to \mathcal{N}_2$ satisfying the following properties: \begin{enumerate}
	\item $(a,b)\in \Sigma_1$ $\Rightarrow (\Phi (a),\Phi(b))\in \Sigma_2$;
	\item $\Phi(a\bullet  b) = \Phi (a)\bullet \Phi(b)$, for all $(a,b)\in \Sigma_1$. 
\end{enumerate} If $\Phi$ is a bijection such that $\Phi$ and $\Phi^{-1}$ are both piecewise homomorphisms, then $\Phi$ is called a \emph{piecewise isomorphism}.\smallskip

Some recent accomplishments in the study of piecewise isomorphisms include the following results: Let ${M}$ and $N$ be von Neumann algebras without direct summands of type $I_1$ or $I_2$. Then, for each bicontinuous piecewise isomorphism \(\Phi : {M}^{++} \to {N}^{++}\), there exists a Jordan $^*$-isomorphism \(\theta : {M} \to {N}\), a bounded linear mapping \(\psi\) from \({M}_{\mathrm{sa}}\) into the self-adjoint part of the centre of $N$, and a non-zero central self-adjoint element  \(c\) in \({M}\) such that $\Phi(a) = e^{\psi(\log a)}\, \theta(a^{c}),$ for all $a \in \mathcal{M}^{++}$ \cite[Theorem 4.3]{Hamhalter2023}. In case that $\Phi : \mathcal{U}({M}) \to \mathcal{U}({N})$ is a bicontinuous piecewise isomorphism, then there exist a non-zero central self-adjoint element $z \in {N}$, a Jordan $^*$-isomorphism $\theta : {M} \to {N}$, and a bounded linear mapping $\psi$ from ${M}_{sa}$ to the self-adjoint part of centre of $N$ such that $\Phi(e^{ita}) = e^{it\psi(a)} e^{izt\theta(a)},$ for all $t \in \mathbb{R},$ $a \in {M}_{sa}$ \cite[Theorem 5.2]{Hamhalter2023}. Consequently, in the case of von Neumann algebras without direct summands of type $I_1$ or $I_2$, the topological groups of positive invertible elements and of unitaries are complete Jordan invariants.\smallskip

As Alfsen and Shultz remarked in the introduction to Chapter~6 of \cite{AlfsenShultz2003}, it is more natural to model quantum mechanics on Jordan algebras rather than on associative algebras; consequently, it is natural to investigate piecewise structures within the framework of Jordan algebras. In the setting of JBW$^*$-algebras, a class formed by not necessarily associative algebras containing all von Neumann algebras in a natural way (see section~\ref{sec:2} below for definitions), we can replace ``commutativity'' with the natural, and well-studied, notion of ``operator commutativity'' and consider piecewise Jordan homomorphisms and isomorphisms in a similar way to that considered above. A very recent study by the second author of this paper assures that for each bicontinuous piecewise Jordan isomorphism $\Phi : \mathfrak{A}^+ \to \mathfrak{B}^+$, where $(\mathfrak{A},\circ)$ and $(\mathfrak{B},\circ)$ are JBW$^*$-algebras admitting no direct summands of type $I_1$ or $I_2$, there exist a Jordan $^*$-isomorphism $\theta : \mathfrak{A} \to \mathfrak{B}$, an invertible central element $c \in \mathfrak{A}$, and a bounded linear mapping $\psi$ from $\mathfrak{A}_{sa}$ to the centre of $\mathfrak{B}$ ($Z(\mathfrak{B})$ in short) such that the identity
\begin{equation*}
\Phi(a) = \exp(\psi(\log a)) \circ \theta(a^c),
\end{equation*} holds for all $a \in \mathfrak{A}^{++}$ \cite[Theorem 3.2]{Hamhalter2025}. It is natural to ask what can be said in the case of piecewise isomorphisms between the sets of unitary elements of two JBW$^*$-algebras admitting no direct summands of type $I_1$ or $I_2$.\smallskip

One of the goals of this paper is to complete the study on piecewise Jordan isomorphisms between unitary sets of two JBW$^*$-algebras. An element $u$ in a unital JB$^*$-algebra is called unitary if it is invertible (in the Jordan sense) with inverse $u^*$. Contrary to the fact that the associative product of two unitaries in a unital C$^*$-algebra is a unitary, the Jordan product of two unitaries is not, in general, a unitary. It is not so obvious that the Jordan product of two operator commuting unitaries is again a unitary.  In section~\ref{sec:2} we revisit the basic results on operator commuting elements in JB$^*$-algebras with the aim of showing that the set, $\mathcal{U}(\mathfrak{A}),$ of all unitary elements in any JBW$^*$-algebra $\mathfrak{A}$ (i.e. a JB$^*$-algebra which is also a dual Banach space) is a piecewise closed subset (cf. Lemma~\ref{lemma_op_comm_unit}).\smallskip

In the main result of section~\ref{sec:3} we show that whenever $\mathfrak{A}$ and $\mathfrak{B}$ are JBW$^*$-algebra without direct summands of type $I_1$ or $I_2$, then for each bijective bicontinuous piecewise Jordan isomorphism $\Phi : \mathcal{U}(\mathfrak{A}) \rightarrow \mathcal{U}(\mathfrak{B})$, there exist a linear Jordan $^*$-isomorphism $\theta: \mathfrak{A} \rightarrow \mathfrak{B}$, a real linear mapping $\beta: \mathfrak{A_{sa}}\rightarrow Z(\mathfrak{B}_{sa})$, and an invertible central element $c \in \mathfrak{B}_{sa}$ such that $$ \Phi(e^{i a}) = e^{i \beta (a)}\circ e^{i  c\circ\theta(a)} = e^{i \beta(a)} \circ \theta \left( e^{i  \theta^{-1}( c )\circ a}\right),$$  for all $a\in \mathfrak{A}_{sa}$. Consequently, under these conditions, $\mathfrak{A}$ and $\mathfrak{B}$ are Jordan $^*$-isomorphic (see Theorem~\ref{theo: 4.3}). The proof of this result relies on the recent study on linear bijections preserving operator commutativity in both directions between JBW$^*$-algebras admitting no central summands of type $I_1$ and $I_2$ presented in \cite{EscolanoPeraltaVillena2025_OpCommut}, and the even more recent Bunce-Wright-Mackey-Gleason theorem for finitely-additive bounded vector-measures on the lattice of projections of a JBW$^*$-algebras without type $I_2$ direct summands established in  \cite{EscolanoPeraltaVillena2025_mes}. The latter result is employed in Theorem~\ref{t corollary of the Mackey-Gleason} to establish the following consequence, which appears to be of independent interest. Let $f: \mathfrak{A}_{sa}\to X$ be a homogeneous mapping which is additive on pairs of operator commuting elements in $\mathfrak{A}$, where $\mathfrak{A}$ is a JBW$^*$-algebra without direct summands of type $I_2$, and $X$ is a real normed space. Then the following statements are equivalent:
\begin{enumerate}[$(a)$]
	\item $f$ is linear and continuous;
	\item $f$ is bounded on the closed unit ball of $\mathfrak{A}_{sa}$. 
\end{enumerate}        
Assuming that $\mathfrak{A}$ is a JBW$^*$-algebra factor not of type $I_2$, we prove in Theorem~\ref{T2} that a bicontinuous bijection $\Phi$ from $\mathfrak{A}$ to another JBW$^*$-algebra $\mathfrak{B}$ is a piecewise Jordan isomorphism if, and only if, there is a Jordan $^*$-isomorphism $\theta: \mathfrak{A}\to \mathfrak{B}$ such that one of the following conditions holds for all $u$ in $\mathcal{U}(\mathfrak{A})$: 
\begin{enumerate}[$(a)$]
	\item  $\Phi(u)=\theta(u)$, for all $u \in \mathcal{U}(\mathfrak{A})$.
	\item  $\Phi(u)=\theta{(u^{-1})}$, for all $u \in \mathcal{U}(\mathfrak{A})$.
\end{enumerate}

Section~\ref{sec: 4 quadratic maps} is devoted to connect our study with previous investigations on quadratic maps between JBW$^*$-algebras due to Y. Friedman and J. Hakeda (see \cite{FriHak88}). To recall the required notions, let us simply note that for each pair of elements $a,b$ in a JB$^*$-algebra $\mathfrak{A}$, the operator $U_{a,b}:\mathfrak{A}\to\mathfrak{A}$ is defined by $U_{a,b} (x) = (a\circ b)\circ c +  (c\circ b)\circ a- (a\circ c)\circ b$. We simply write $U_{a}$ for $U_{a,a}$. If a C$^*$-algebra $A$ is regarded as a JB$^*$-algebra with respect to its natural Jordan product, $a\circ b = \frac12 (a b + b a )$, we have $U_{a,b} (c) = \frac12 ( a c b + b c a )$ ($a,b,c\in A$).  Every JB$^*$-algebra can be regarded as a JB$^*$-triple in the sense of \cite{Kaup83} with respect to the triple product $\{a,c,b\}= U_{a,b} (c^*)$ ($a,b,c\in \mathfrak{A}$). Let $\mathfrak{A}$ and $\mathfrak{B}$ be JBW$^*$-algebras such that $\mathfrak{A}$ admits no direct summands of type $I_1$. It follows from a result by Friedman and Hakeda (cf. \cite[Proposition 2.14, Theorem 2.15]{FriHak88}) that every bijection $\Phi: \mathfrak{A} \to \mathfrak{B}$ satisfying 
$$\Phi(\{a,b,a\})\!=\Phi\left(U_a(b^*)\right) =U_{\Phi(a)}(\Phi(b)^*))= \{\Phi(a),\Phi(b),\Phi(a)\},$$ for all $a, b \in \mathfrak{A}$ is a real linear isometry. \smallskip

In Theorem~\ref{t FriedmanHakeda quadratic and additive on oc elements} we prove the following variant of the just commented result: Let $\mathfrak{A}$ and $\mathfrak{B}$ be JBW$^*$-algebras such that $\mathfrak{A}$ admits no direct summands of type $I_2$. Suppose $\Phi : \mathfrak{A}_{sa}\to \mathfrak{B}_{sa}$ is a mapping satisfying the following hypothesis: 
\begin{enumerate}[$(a)$]
\item $\Phi$ preserves quadratic products on operator commuting elements, equivalently, $\Phi\left(U_a(b)\right) =U_{\Phi(a)}(\Phi(b))$ whenever $a$ and $b$ are operator commuting elements in $\mathfrak{A}_{sa}$. 
\item  $\Phi $ is additive on operator commuting elements, that is, $\Phi(a+b) = \Phi(a) + \Phi(b),$ whenever $a$ and $b$ are operator commuting elements in $\mathfrak{A}_{sa}$.
\end{enumerate}
Then, $\Phi(\unit)$ is a tripotent in $\mathfrak{B}_{sa}$, the Peirce-$2$ subspace $\mathfrak{B}_2 (\Phi(\unit))$ is naturally equipped with a structure of JBW$^*$-algebra whose self adjoint part contains the image of $\Phi$, and makes the mapping $\Phi: \mathfrak{A}_{sa} \to \mathfrak{B}_2 (\Phi(\unit))_{sa}$ a unital (real linear) Jordan homomorphism. Furthermore, if $\Phi$ is bijective, we can conclude that $\Phi (\unit)$ is a central symmetry in $\mathfrak{B}$ and $\Phi$ is an isometric real linear Jordan isomorphism from $\mathfrak{A}_{sa}$ to $\mathfrak{B}_2 (\Phi(\unit))_{sa}$.\smallskip

Finally, in Example~\ref{couterexample spin} we show that the hypothesis concerning type $I_2$ summands cannot be relaxed in Theorems~\ref{t corollary of the Mackey-Gleason} and \ref{t FriedmanHakeda quadratic and additive on oc elements}.

\section{Preservers of Jordan products on piecewise closed subsets of JB\texorpdfstring{$^*$}{*}-algebras}\label{sec:2}

JB$^*$-algebras, whose definition was foreseen by Kaplansky in 1976, are precisely the Jordan Banach $^*$-algebras that play for Jordan structures the role that C$^*$-algebras play for associative operator algebras. These algebras provide an appropriate non-associative framework in which one can do functional analysis, spectral theory, and state-space geometry in a way strongly parallel to the C$^*$-algebraic setting. A \emph{JB$^*$-algebra} is a complex Banach space $\mathfrak{A}$ together with a bilinear and commutative product (denoted by $a\circ b$) and an algebra involution ($a\mapsto a^*$) satisfying the following axioms: 
\begin{enumerate}[$(1)$]
	\item $(a\circ b)\circ b^2  = (a\circ b^2)\circ b\ $ \ \ ($\forall a,b\in \mathfrak{A}$) \ \ (\emph{Jordan identity}),
	\item $\left\| U_{a} (a^*) \right\| = \|a\|^3$ \ \ ($\forall a\in \mathfrak{A}$), where $U_a (b) = 2 (a\circ b)\circ a - a^2 \circ b$ \ \ ($\forall a,b\in \mathfrak{A}$). 
\end{enumerate} This definition is equivalent to say that the self-adjoint part $\mathfrak{A}$ (which is denoted by $\mathfrak{A}_{sa}$) is a real JB-algebra in the usual sense (see \cite{HOS,Wright1977}). The basic theory of JB- and JB$^*$-algebras can be consulted in \cite{AlfsenShultz2003, HOS, Cabrera, Wright1977}, and the introduction of the recent papers \cite{EscolanoPeraltaVillena2025_mes,Hamhalter2025}.\smallskip 

We briefly recall here that a JB$^*$-algebra is called unital if there exists an element $\mathbf{1}\in \mathfrak{A}$ such that $\mathbf{1}\circ a = a\circ \mathbf{1} = a,$ for all $a\in \mathfrak{A}$. Every C$^*$-algebra is a JB$^*$-algebra if we consider the Jordan product given by $a\circ b = \frac12 (a b + b a)$. However, the class of JB$^*$-algebras is strictly larger than the class of C$^*$-algebras. A JBW$^*$-algebra is a JB$^*$-algebra which is also a dual Banach space.  The second dual, $\mathfrak{A}^{**}$, of a JB$^*$-algebra $\mathfrak{A}$ is a JBW$^*$-algebra containing the original structure as a JB$^*$-subalgebra \cite[4.4.3 Theorem 4.4.3]{HOS}, and each JBW$^*$-algebra is unital, admits a unique isometric predual and its product is separately weak$^*$-continuous \cite[Corollary 4.1.6, Lemma 4.1.7,  and Theorem 4.4.16]{HOS}.\smallskip

On every JB$^*$-algebra $\mathfrak{A}$ we can consider the triple product defined by $\{a,b,c\} = (a\circ b^*)\circ c + (c\circ b^*)\circ a - (a\circ c)\circ b^*$. This triple product induces a structure of JB$^*$-triple on $\mathfrak{A}$ (cf. \cite{Kaup83}). In case that a C$^*$-algebra $A$ is regarded with its natural Jordan product, the derived triple product satisfies $\{a,b,c\} = \frac12 (a b^* c + c b^* a)$ for all $a,b,c\in A$.\smallskip

As commented in the introduction, in the seeking of algebraic invariants of C$^*$-algebras authors considered special subsets that, despite not being subspaces, nor closed under products, contain the products of certain pairs of their elements. We begin by observing how certain subsets of unital JB$^*$-algebras are piecewise closed for the relation of commeasurability given by operator commutativity.\smallskip

The product of a Jordan algebra $\mathfrak{A}$ is, by definition, commutative, so every pair of elements in $\mathfrak{A}$ commute if we only employ the ``usual'' sense. However, if we assume that an associative algebra $A$ is equipped with its natural Jordan product $a\circ b = \frac12 ( ab + ba )$, it can be easily checked that if $a,b\in A$ commute with respect to the associative product (i.e., $ a b = b a$), then the corresponding Jordan multiplication operators $M_a (x) = a\circ x$ and $M_b(x)= b\circ x$ ($x\in \mathfrak{A}$) commute. According to the standard sources, elements $a$ and $b$ in a Jordan algebra $\mathfrak{A}$ are said to \emph{operator commute} if the operators $M_a, M_b$ commute (i.e., $(a\circ c)\circ b = a \circ (c\circ b)$ for every $c \in \mathfrak{A}$). For example, the Jordan identity is equivalent to say that every element $a$ in a JB$^*$-algebra  operator commute with its square. The general lacking of associativity motivates the use of the \emph{associator} of three elements $a,b,c \in \mathfrak{A}$ defined by $[a,c,b] := [M_a, M_b](c) =  (a\circ c) \circ b - a \circ (c\circ b)$. Clearly, $a$ and $b$ operator commute in $\mathfrak{A}$ if, and only if, $[a,\mathfrak{A},b]=0.$ \smallskip

Contrary to what one might expect naturally, in a general Jordan algebra, the statements ``$a$ and $b$ operator commute'' and ``$a$ and $b$ generate a commutative and associative algebra'' are completely independent (cf. \cite[2.5.1 and Example 2.5.2]{HOS}). It is known (see \cite[Theorem]{vandeWet2020}) that in the case that $a$ and $b$ are self-adjoint elements in a JB$^*$-algebra $\mathfrak{A}$ the following statements are equivalent: 
\begin{enumerate}[$(a)$]
	\item $a$ and $b$ operator commute.
	\item $a$ and $b$ generate an associative JB$^*$-algebra, equivalently, a commutative C$^*$-algebra.
	\item $a$ and $b$ generate an associative JB$^*$-algebra of mutually operator commuting elements.
\end{enumerate}  Furthermore, in such a case $a$ and $b$ commute in the usual sense as elements in any C$^*$-algebra containing the JB$^*$-subalgebra of $\mathfrak{A}$ generated by $a$ and $b$ as a Jordan self-adjoint subalgebra \cite[Proposition 1]{Topping65}.\smallskip

The \emph{centre} of a JB$^*$-algebra $\mathfrak{A}$ is formed by the elements $z\in \mathfrak{A}$ that operator commute with every element in $\mathfrak{A}$. The symbol $Z(\mathfrak{A})$ will stand for the centre of $\mathfrak{A}$, and its elements are called central. The centre of a JB$^*$-algebra $\mathfrak{A}$ is a commutative C$^*$-algebra, and contains the identity of $\mathfrak{A}$ in case that the latter exists (see \cite[Proposition 1.52]{AlfsenShultz2003}. The centre of a JBW$^*$-algebra is a commutative von Neumann (see \cite[Proposition 2.36]{AlfsenShultz2003}, \cite{Edwards80}). If $Z(\mathfrak{A}) = \mathbb{C}\  \unit$ we say that $\mathfrak{A}$ is JBW$^*$-factor.\smallskip 
    
An element $a$ in a unital JB$^*$-algebra $\mathfrak{A}$ is called \emph{invertible} if there exists a (unique) $b \in \mathfrak{A}$ satisfying $a\circ b = \unit$ and $a^2\circ b = a$. The element $b$ is known as the inverse of $a$, and is denoted by $a^{-1}$ (cf. \cite[3.2.9]{HOS} or \cite[\S 4.1.1]{Cabrera}). The (\emph{Jordan}-)\emph{spectrum} of an element $a$ in $\mathfrak{A}$ (denoted by J-$\sigma(a)$) is the set of all $\lambda\in \mathbb{C}$ for which $a-\lambda \textbf{1}$ is not invertible. As in the setting of associative Banach algebras, J-$\sigma(a)$ is a non-empty compact subset of the complex plane (see \cite[Theorem 4.1.17]{Cabrera}. \smallskip


We can now present some examples of piecewise closed subsets of JB$^*$-algebras. An element $p$ in a JB$^*$-algebra $\mathfrak{A}$ is called a \emph{projection} if $p=p^*= p^2$. We shall write Proj$(\mathfrak{A})$ for the set of all projections in $\mathfrak{A}$. An element $a\in \mathfrak{A}$ is said to be \emph{positive} ($a\geq 0$ in short) if $a^* = a$ and its (Jordan) spectrum in contained in $\mathbb{R}_0^+$, equivalently, $a$ is the square of a hermitian element (see \cite[3.3.3]{HOS} and \cite[\S 3.1.27]{Cabrera}). The symbols $\mathfrak{A}^+$ and $\mathfrak{A}^{++} = \mathfrak{A}_{sa}^{++}$ will stand for the sets of all positive elements and all positive and invertible elements in $\mathfrak{A}$, respectively. 
It naturally follows from the above properties that for each unital JB$^*$-algebra $\mathfrak{A}$, the sets Proj$(\mathfrak{A}),$ $\mathfrak{A}^{+},$ and $\mathfrak{A}^{++}$ are piecewise closed subsets of $\mathfrak{A}$.
\smallskip

Along this note, piecewise closed subsets of JB- and JB$^*$-algebras are also called \emph{piecewise Jordan closed subsets}, while piecewise homomorphisms will be also called \emph{piecewise Jordan homomorphisms}.\smallskip

We say that two projections $p,q\in \mathfrak{A}$ are \emph{orthogonal} ($p\perp q$ in short) if $p\circ q =0$. It is known that $p$ and $q$ operator commute if they are orthogonal (cf. \cite[Theorem]{vandeWet2020} or \cite[Lemma 2.5.5]{HOS}).\smallskip 

The JB$^*$-subalgebra generated by a single hermitian element $a$  in a JB$^*$-algebra $\mathfrak{A}$ is isometrically Jordan $^*$-isomorphic to the commutative C$^*$-algebra of all continuous complex-valued functions on the (Jordan-)spectrum of $a$, and $a$ can be identified with the inclusion of its spectrum in $\mathbb{C}$ (cf. \cite[\S 3.2]{HOS}, \cite[Corollary 1.19 and Definition 1.20]{AlfsenShultz2003} or \cite[Corollary 4.1.72 and Corollary 4.1.74]{Cabrera}). We can therefore define a continuous functional calculus at the element $a$. It will be employed along this note without explicit mention. \smallskip  

Let $A$ and $B$ be two C$^*$-algebras, and let $\mathcal{N}$ be piecewise closed subset of $A$. Suppose that $\Delta: \mathcal{N} \to B$ is just a mapping which is multiplicative on pairs of commuting elements. If $a,b$ are two commuting elements in $\mathcal{N}$, their images, $\Delta (a)$ and $\Delta (b),$ automatically commute in $B$. The conclusion is not so obvious in the case of JB$^*$-algebras. We begin with piecewise closed subsets of JB-algebras. \smallskip

Henceforth, given a subset $\mathcal{S}$ of a Banach space $X$, we shall write $\overline{span} \left( \mathcal{S} \right)$ for the closed linear span of $X$.  

\begin{proposition}\label{prop:asocc_def} Let $\mathfrak{A}$ and $\mathfrak{B}$ be a pair of JB$^*$-algebras, and let $\mathcal{N}_1 \subset \mathfrak{A}_{sa}$ be a piecewise Jordan closed subset. Suppose $ \Delta: \mathcal{N}_1 \longrightarrow \mathfrak{B}_{sa}$ is a mapping preserving Jordan products on pairs of operator commuting elements in $\mathcal{N}_1$. Then, for every pair of operator commuting elements $a$ and $b$ in $\mathcal{N}_1$, the elements $\Delta(a)$ and $\Delta(b)$ operator commute in $\mathfrak{B}$. 
\end{proposition}

\begin{proof} Fix two operator commuting elements $a,b\in \mathcal{N}_1$. Let JB$(a,b)$ denote the JB-subalgebra of $\mathfrak{A}_{sa}$ generated by $a$ and $b$. As we commented above, since $a$ and $b$ operator commute, the JB-algebra JB$(a,b)$ is an associative algebra (cf. \cite[Theorem]{vandeWet2020}). The assumptions on $\mathcal{N}_1$ guarantee that JB$(a,b)\cap  \mathcal{N}_1$ also is a piecewise closed subset of $\mathfrak{A}_{sa}$ with respect to operator commutativity. Observe that, by construction and hypotheses, JB$(a,b)\cap  \mathcal{N}_1$ is closed for Jordan products of its elements, and so the norm-closure of its real-linear span is a JB-subalgebra of $\mathfrak{A}_{sa}$. Since $a,b\in \hbox{JB}(a,b)\cap  \mathcal{N}_1$, it follows that $\overline{span} \left( \hbox{JB}(a,b)\cap  \mathcal{N}_1 \right) = \hbox{JB}(a,b).$ Having in mind these properties, we deduce that for any two elements $c,d$ in $\hbox{JB}(a,b) \cap \mathcal{N}_1$ we have $c\circ d\in \hbox{JB}(a,b) \cap \mathcal{N}_1$ and
    \begin{equation}\label{eq: assoc_1}
        \Delta(c)\circ \Delta(d) = \Delta(c\circ d) \in \Delta \left( \hbox{JB}(a,b)\cap \mathcal{N}_1 \right). 
    \end{equation} Consequently, $\Delta \left( \hbox{JB}(a,b)\cap \mathcal{N}_1 \right)$ is closed for Jordan products of its elements. Now, let $\mathcal{C}$ denote the JB-subalgebra of $\mathfrak{B}_{sa}$ generated by $\Delta(\hbox{JB} (a,b)\cap \mathcal{N}_1) $. We shall show that $\mathcal{C}$ is associative. To prove this fact it suffices to show that $\Delta(\hbox{JB}(a,b)\cap \mathcal{N}_1)$ is associative. Given $x,y,z \in \hbox{JB}(a,b)\cap \mathcal{N}_1$, it follows from \eqref{eq: assoc_1} and the associativity of $\hbox{JB}(a,b)$ that 
    $$
     \begin{aligned}
        (\Delta(x)\circ \Delta(y))\circ \Delta(z) &= \Delta(x\circ y)\circ \Delta(z) = \Delta((x\circ y)\circ z) = \Delta(x \circ (y\circ z))  \\
        & = \Delta(x)\circ \Delta(y\circ z)  = \Delta(x)\circ (\Delta(y)\circ \Delta(z)), 
    \end{aligned}
    $$
 which gives the desired statement. Note now that the JB-subalgebra of $\mathfrak{B}_{sa}$ generated by $\Delta(a)$ and $\Delta(b)$ must be contained in $\mathcal{C}$, and hence it must be an associative subalgebra of $\mathfrak{B}_{sa}$.  The main theorem in \cite{vandeWet2020} assures that $\Delta(a)$ and $\Delta(b)$ operator commute. 
\end{proof}

In order to present additional examples of piecewise closed subsets of JB$^*$-algebras, we recall that an element $u$ in a unital JB$^*$-algebra $\mathfrak{A}$ is said to be a \emph{unitary} if it is invertible (i.e. there exists, a unique, $b\in \mathfrak{A},$ called the \emph{inverse} of $a$, such that $a \circ b =\mathbf{1}$ and $a^2 \circ b = a$) and its inverse coincides with $u^*$ (cf. \cite[\S 4.1.1, Theorem 4.2.28]{Cabrera} or \cite[\S 2]{CuetoPeralta2023}). The set of all unitaries in $\mathfrak{A}$ will be denoted by $\mathcal{U}(\mathfrak{A})$. An element $s$ in $\mathfrak{A}_{sa}$ is called a \emph{symmetry} if $s^2 = \unit$. The set of all symmetries in $\mathfrak{A}$ will be denoted by $\hbox{Symm}(\mathfrak{A})$. 
Basic properties of the set of unitary elements are gathered in \cite[Lemma 2.1]{CuetoPeralta2023}. For the moment being, the reader should be warned that the Jordan product of two unitaries might not be another unitary, however $U_u (v)$ is a unitary whenever $u$ and $v$ are.\smallskip

We continue with a technical lemma on operator commutativity for unitary elements in JBW$^*$-algebras. 

\begin{lemma}{\label{lemma_op_comm_unit}} Let $u,v$ be unitary elements in a JBW$^*$-algebra $\mathfrak{A}$. Then the following conditions are equivalent: 
    \begin{enumerate}[$(a)$]
        \item $u$ and $v$ operator commute in $\mathfrak{A}$, i.e. $u$ and $v$ are commeasurable;
         \item There exist $h,k\in \mathfrak{A}_{sa}$ such that $h$ and $k$ operator commute in $\mathfrak{A}$, $u = e^{i h}$ and $v = e^{ik}$;
        \item $u$ operator commutes with $v$ and $v^*$;
        \item The JB$^*$-subalgebra of $\mathfrak{A}$ generated by $u$ and $v$ is associative.   
    \end{enumerate}
\end{lemma}

\begin{proof} Since $\mathfrak{A}$ is a JBW$^*$-algebra, we can find hermitian elements $h_1,k_1\in \mathfrak{A}$ such that $u = e^{i h_1}$ and $v = e^{i k_1}$ (cf. \cite[Remark 3.2]{CuetoPeralta2023}). Note that, even assuming that $u$ and $v$ operator commute, we do not have direct arguments to show that $h_1$ and $k_1$ operator commute in $\mathfrak{A}$.\smallskip
	
$(a) \Rightarrow (b)$ Let $\mathfrak{A}_{h_1,k_1}$ and $\mathfrak{A}_{u,v}$ denote the JB$^*$-subalgebras of $\mathfrak{A}$ generated by $\{h_1,k_1\}$ and $\{u,v\}$, respectively. The weak$^*$-closure of $\mathfrak{A}_{u,v}$ in $\mathfrak{A}$, denoted by $\overline{\mathfrak{A}_{u,v}}^{w^*}$, is a JBW$^*$-subalgebra of $\mathfrak{A}$. By the the Shirshov-Cohn theorem (see \cite[Theorem 2.4.14]{HOS} and \cite{Wright1977}), there exists a unital C$^*$-algebra $A,$ with unit $\mathbf{1}$,  containing $\mathfrak{A}_{h_1,k_1}$ as a JB$^*$-subalgebra. Clearly $u,v\in \mathfrak{A}_{u,v}\subseteq  \mathfrak{A}_{h_1,k_1}\subseteq A$ as JB$^*$-subalgebras. Since $u$ and $v$ operator commute in $\mathfrak{A}$, we deduce from \cite[Proposition 1.2]{EscolanoPeraltaVillena2025_OpCommut} that $u$ and $v$ commute as elements in $A$. Let $A_{u,v}$ denote the C$^*$-subalgebra of $A$ generated by $u$ and $v$, which must be commutative because $u$ and $v$ commute in $A$. Since $\mathfrak{A}_{u,v}\subseteq {A}_{u,v}$, we deduce that $\mathfrak{A}_{u,v}$ must be associative, and consequently $\overline{\mathfrak{A}_{u,v}}^{w^*}$ is an associative JBW$^*$-algebra, and hence a commutative von Neumann algebra. By \cite[Proposition 4.4.10]{KadRingVolI} (or by \cite[Remark 3.2]{CuetoPeralta2023}), there exists $h,k\in \left(\overline{\mathfrak{A}_{u,v}}^{w^*}\right)_{sa}\subseteq \mathfrak{A}_{sa}$ such that $u = e^{i h}$ and $v = e^{i k}$. Having in mind that $h$ and $k$ are self-adjoint elements of $\mathfrak{A}$ contained in the associative JBW$^*$-algebra $\overline{\mathfrak{A}_{u,v}}^{w^*}$, the commented result by van de Wetering (cf. \cite[Theorem]{vandeWet2020}) assures that $h$ and $k$ operator commute in $\mathfrak{A}$ as desired.\smallskip 
	
$(b) \Rightarrow (c)$ Under the assumptions in $(b)$, a    
new application of \cite[Theorem]{vandeWet2020} assures that any two powers of $h$ and $k$ operator commute in $\mathfrak{A}$. Observe now that $$\left[ M_{u}, M_{v^*} \right] = \left[ M_{e^{ih}}, M_{e^{-ik}} \right]  = \sum_{j,n=0}^{\infty} \frac{i^{j-n}}{j! n!} \left[M_{h^{j}}, M_{k^{n}}\right]=0,$$ which proves that $u$ and $v^*$ operator commute. Similar arguments lead to the conclusion that $u$ and $v$  operator commute.\smallskip

The implications $(c) \Rightarrow (a)$ and $(b) \Rightarrow (d)$ are clear (see \cite[Theorem]{vandeWet2020} for the latter implication). \smallskip

$(d) \Rightarrow (b)$ Suppose now that the JB$^*$-subalgebra of $\mathfrak{A}$ generated by $u$ and $v$ (denoted by $\mathfrak{A}_{u,v}$) is associative. Therefore, $\mathfrak{A}_{u,v}$ is a commutative C$^*$-algebra, and its weak$^*$-closure in $\mathfrak{A}$ is a commutative von Neumann algebra. As in the proof of $(a)\Rightarrow (b)$, there exists $h_1,k_1\in (\overline{\mathfrak{A}_{u,v}}^{w^*})_{sa}$ operator commuting in $\mathfrak{A}$ such that $u = e^{i h_1}$ and $v = e^{i k_1}$. 
\end{proof}

\begin{remark}\label{remark exponential of i hermitians} In a unital JB$^*$-algebra $\mathfrak{A},$ the set $\{e^{i h} : h\in \mathfrak{A}_{sa}\}$ can be a strict subset of $\mathcal{U} (\mathfrak{A})$ (cf. \cite[Remark 3.2]{CuetoPeralta2023}). Lemma~\ref{lemma_op_comm_unit} proves the equivalence of the following statements for all $u,v\in \mathcal{U} (\mathfrak{A})$ by just considering $\mathfrak{A}$ inside $\mathfrak{A}^{**}$:
    \begin{enumerate}[$(1)$]
	\item $u$ and $v$ operator commute in $\mathfrak{A}$, i.e. $u$ and $v$ are commeasurable;
	\item There exist $h,k\in \mathfrak{A}^{**}_{sa}$ such that $h$ and $k$ operator commute in $\mathfrak{A}^{**}$, $u = e^{i h}$ and $v = e^{ik}$;
	\item $u$ operator commute with $v$ and $v^*$;
	\item The JB$^*$-subalgebra of $\mathfrak{A}$ generated by $u$ and $v$ is associative.   
\end{enumerate}	
	
Actually the following statements are equivalent for all $h,k\in \mathfrak{A}_{sa}$:
\begin{enumerate}[$(a)$]
	\item $e^{i t h}$ and $e^{it k}$ operator commute in $\mathfrak{A}$ for all $t\in \mathbb{R}$;
	\item $h$ and $k$ operator commute;
	\item The JB$^*$-subalgebra of $\mathfrak{A}$ generated by $h$ and $k$ is associative;	
	\item $U_{e^{it h}} \left(e^{2 i s k}\right) = \{e^{it h}, e^{-i2 s k }, e^{i t h}\} = \{ e^{i s k}, e^{- i2 t h}, e^{i s k}\} = U_{e^{i s k}} \left(e^{2 i t h}\right),$ for all $s,t$ in $\mathbb{R}$.  
\end{enumerate} The equivalence $(b)\Leftrightarrow (c)$ follows from \cite[Theorem]{vandeWet2020}.  The just quoted result also assures that if $h$ and $k$ operator commute, it holds that any two powers of $t h$ and $t k$ operator commute in $\mathfrak{A}$ for all real $t$, and thus $e^{i t h}$ and $e^{it k}$ operator commute in $\mathfrak{A}$ for all $t\in \mathbb{R}$. That is $(b)\Rightarrow (a)$.\smallskip

Assuming that $(a)$ holds we deduce that $$\displaystyle [M_h,M_k] = \left[ M_{ \lim_{t\to 0} \frac{e^{i t k}-\unit}{t}},  M_{\lim_{t\to 0} \frac{e^{i t h}-\unit}{t}}  \right] = \lim_{t\to 0} \frac{1}{t^2} \left[ M_{e^{i t k}}- M_{\unit},  M_{e^{i t h}} - M_{\unit}  \right] =0,$$ and thus $h$ and $k$ operator commute in $\mathfrak{A}$. Therefore,  $(a)\Rightarrow (b)$.\smallskip

Clearly, $(b) \Rightarrow (d)$. To prove $(d) \Rightarrow (b)$, let us fix $s \in \mathbb{R},$ and define two maps $\omega_1, \omega_2: \mathbb{C}\to \mathfrak{A}$ by 
$$ \omega_1(\lambda) = U_{e^{\lambda h}} (e^{2 i s k}),\hbox{ and } \omega_2(\lambda) = U_{e^{i s k}}(e^{2 \lambda h}), \ (\lambda \in \mathbb{C}).$$ Clearly, $\omega_1$ and $\omega_2$ are holomorphic maps which coincide on the imaginary axis by assumptions, and hence they both define the same function on $\mathbb{C}$. The arbitrariness of $s\in \mathbb{R}$ assures that  $$U_{e^{\lambda h}} (e^{2 i s k}) = U_{e^{i s k}}(e^{2 \lambda h}), \ (\lambda \in \mathbb{C}, s\in \mathbb{R}),$$ that is, if we fix $\lambda\in \mathbb{C}$, the holomorphic mappings $\tilde{\omega}_1, \tilde{\omega}_2: \mathbb{C}\to \mathfrak{A},$  $ \tilde{\omega}_1(\mu) = U_{e^{\lambda h}} (e^{2 \mu k}),$ $\tilde{\omega}_2(\mu) = U_{e^{\mu k}}(e^{2 \lambda h})$ coincide on the imaginary axis, and thus on the whole complex plane. Therefore, $$U_{e^{\lambda h}} (e^{2 \mu k}) = U_{e^{\mu k}}(e^{2 \lambda h}), \hbox{ for all } (\lambda,\mu \in \mathbb{C}),$$ and consequently $U_{e^{\frac{t h}{2}}}(e^{t k}) = U_{e^{\frac{t k}{2}}}(e^{t h}).$ Proposition 1.1 in \cite{Hamhalter2025} assures that $e^{t h}$ and $e^{t k}$ operator commute in $\mathfrak{A}_{sa}$ (and hence on $\mathfrak{A}$) for all real $t$. Similar arguments to those given above show that $h$ and $k$ operator commute.  
\end{remark}

\begin{remark}\label{remark JB*-subalgebra gen by a unitary}  Let $u$ be a unitary element in a unital JB$^*$-algebra $\mathfrak{A}$. Although, we cannot always write $u = e^{i h}$ for some $h\in \mathfrak{A}_{sa}$ (see the previous remark), we can always find $h\in \mathfrak{A}_{sa}^{**}$ such that $u = e^{i h}$ in $\mathfrak{A}^{**}$. The JB$^*$-subalgebra $\mathfrak{A}^{**}(h,\unit)$ of $\mathfrak{A}^{**}$ is associative, and thus JB$^*$-isomorphic to a commutative von Neumann algebra. Clearly $\mathfrak{A}^{**}(h,\unit)$ contains $u$ and the whole JB$^*$-subalgebra of $\mathfrak{A}$ generated by $u$ and $\unit$. So, the JB$^*$-subalgebra of $\mathfrak{A}$ generated by $u$ and $\unit$ is associative.
\end{remark}

It follows from the above Lemma~\ref{lemma_op_comm_unit} that, for each JBW$^*$-algebra $\mathfrak{A},$ the set $\mathcal{U}(\mathfrak{A})$ is piecewise closed.  We state next a version of Proposition~\ref{prop:asocc_def} for the set of unitary elements in a JBW$^*$-algebra.

\begin{proposition}\label{prop:asocc_def_unit} Let $\mathfrak{A}$ and $\mathfrak{B}$ be JBW$^*$-algebras.  Suppose $ \Delta: \mathcal{U}(\mathfrak{A}) \longrightarrow \mathfrak{B}$ is a mapping preserving Jordan products on operator commuting elements in $\mathcal{U}(\mathfrak{A})$. Then, for every pair of operator commuting elements $u,v $ in $U(\mathfrak{A}),$ the corresponding images $\Delta(u)$ and $\Delta(v)$ operator commute in $\mathfrak{B}$. 
\end{proposition}

\begin{proof} Keeping the notation above, let the symbols JB$^*(u,v)$ and JB$^*(\Delta(u), \Delta(v))$ stand for the  JB$^*$-algebras of $\mathfrak{A}$ and $\mathfrak{B}$ generated by $\{u,v,\mathbf{1}\}$ and $\{\Delta(u), \Delta(v), \mathbf{1}\}$, respectively. Since $u$ and $v$ operator commute, Lemma~\ref{lemma_op_comm_unit} implies that JB$^*(u,v)$ is associative, and hence a commutative and unital C$^*$-algebra.\smallskip

Observe that the set $\mathcal{U} \left( \hbox{JB}^*(u,v) \right) = \mathcal{U} (\mathfrak{A})\cap \hbox{JB}^*(u,v)\subseteq \mathcal{U} (\mathfrak{A})$ is self-adjoint and closed for Jordan products, and contains $u,v,$ and $\mathbf{1}$. Consequently, $\hbox{JB}^*(u,v) = \overline{span} \left( \hbox{JB}^*(u,v)\cap \mathcal{U} (\mathfrak{A}) \right)$.\smallskip

Given $w_1,w_2,w_3\in \mathcal{U} \left( \hbox{JB}^*(u,v) \right)$, the hypotheses on $\Delta$ guarantee that $$\Delta (w_1) \circ \Delta(w_2) = \Delta (w_1\circ w_2) \in \Delta \left(\mathcal{U} \left( \hbox{JB}^*(u,v) \right)\right),$$ and $$\begin{aligned}
 \Delta (w_1) &\circ (\Delta(w_2) \circ \Delta(w_3))= \Delta (w_1) \circ \Delta(w_2 \circ w_3) = \Delta (w_1 \circ (w_2 \circ w_3)) \\&= \Delta ((w_1 \circ w_2) \circ w_3) = \Delta (w_1\circ w_2) \circ \Delta( w_3) = (\Delta (w_1)\circ \Delta ( w_2) )\circ \Delta( w_3).
\end{aligned}$$ All together proves that $\Delta \left(\mathcal{U} \left( \hbox{JB}^*(u,v) \right)\right)$ is closed for the Jordan product of its elements and the Jordan product acts associatively on any three of them. Since clearly $\Delta(u), \Delta(v)\in \Delta \left(\mathcal{U} \left( \hbox{JB}^*(u,v) \right)\right),$ it follows that the JB$^*$-subalgebra of $\mathfrak{B}$ generated by $\Delta(u), \Delta(v)$ and the unit must be associative. Finally, Lemma~\ref{lemma_op_comm_unit} assures that $\Delta(u)$ and  $\Delta(v)$ operator commute.
\end{proof}

\section{Preservers of Jordan products on the set of unitaries in a JBW\texorpdfstring{$^*$}{*}-algebra}\label{sec:3}

We have already seen in the previous section that the set of unitary elements in a JBW$^*$-algebra is a piecewise closed subset with respect to operator commutativity. In this section we study piecewise Jordan homomorphisms between certain sets of unitary elements in unital JB$^*$-algebras. \smallskip 

By the Gelfand--Naimark theorem, every unitary $u$ in a unital $C^{*}$-algebra $A$ can be viewed as a unitary element in the algebra $B(H)$ of all bounded linear operators on a complex Hilbert space $H$, in such a way that $u$ itself acts as a unitary operator on $H$. Consequently, one-parameter unitary groups in $A$ fall under the hypotheses of well-known results such as Stone's one-parameter theorem.
However, unitary elements in a unital JB$^{*}$-algebra $\mathfrak{A}$ can not be always regarded as unitaries on some complex Hilbert space $H$. Thanks to the useful result in \cite[Theorem 3.1]{CuetoPeralta2023} (see also \cite{GarPe2021}), we now have access to a genuine Jordan analogue of Stone's theorem for uniformly continuous one-parameter unitary groups for unital JB$^{*}$-algebras.

\begin{proposition}\label{1.2} Let $\mathfrak{A}$ and $\mathfrak{B}$ be unital JB$^*$-algebras. Let $\mathcal{N}$ be a piecewise closed subset of $\mathfrak{A}$ containing $\exp(i \mathfrak{A}_{sa})$. Suppose finally that $$\Phi: \mathcal{N}\longrightarrow \mathcal{U}(\mathfrak{B})$$ is a continuous piecewise Jordan homomorphism. Then there exists a mapping $f: \mathfrak{A}_{sa}\to \mathfrak{B}_{sa}$ which is homogeneous, additive on operator commuting elements, preserves operator commutativity, and satisfies the following identity:
    	\[  \Phi(e^{ita})=e^{if(a)t} \,, \mbox{ for all } a\in \mathfrak{A}_{sa} \mbox{ and } t\in \mathbb{R}\,.               \] 
\end{proposition}

\begin{proof} Let us pick $a\in \mathfrak{A}_{sa}$, and define a mapping 
	\[q: \mathbb{R}\to \exp(i \mathfrak{A}_{sa})\subseteq \mathcal{N},\ \ q(t)= e^{ita}.    \] 
Then the range of $q$ lies in the associative JB$^*$-subalgebra generated by $a$ and the unit element, and hence $q(\mathbb{R})$ consists of operator commuting unitary elements in $\mathcal{N}$ (cf. Remark~\ref{remark exponential of i hermitians}). Let us now consider the mapping $Q = \Phi q :\mathbb{R}\to \mathcal{U}(\mathfrak{B}).$ The properties of the maps $\Phi$ and $q$ guarantee that $$Q (t+ s) = \Phi (q(s) \circ q(t)) = \Phi (q(s)) \circ \Phi ( q(t)) = Q(s)\circ Q(t),$$ and hence $Q$ is a continuous one-parameter group of unitary elements in $\mathfrak{B}$. By considering that the JB$^*$-subalgebra generated by the set $\set{\Phi(e^{ita})}{t\in \mathbb{R}}$ must be a commutative C$^*$-algebra (cf. Proposition~\ref{prop:asocc_def_unit}), or by applying \cite[Theorem 3.1]{CuetoPeralta2023}, we deduce the existence of a unique $f(a)\in \mathfrak{B}_{sa}$ such that $\Phi(e^{iat})= Q(t)=   e^{i t f(a)}$, for all $t\in \mathbb{R}$. We have therefore defined a function $f: \mathfrak{A}_{sa}\to \mathfrak{B}_{sa}$ satisfying the final identity in the conclusions. We shall prove next the other properties of $f$.\smallskip
 
Let us take two  operator commuting elements $a,b\in \mathfrak{A}_{sa}$. By construction, we have
	$$ e^{i t f(a+b)}=\Phi(e^{i t (a+b)}) = \Phi(e^{i t a})\circ \Phi(e^{i t b}) =e^{i t f(a)}  \circ  e^{i t f(b)} = e^{i t (f(a)+f(b))}\ \ \ (t\in \mathbb{R}),$$ which guarantees that $ f(a+b)=f(a)+f(b),$ that is, $f$ is additive on operator commuting elements. The real-$1$-homogeneity of $f$ is even easier to check.\smallskip
	  
The hypotheses on $\Phi$ imply that, for $a$ and $b$ as above,  $\Phi(e^{ita})= e^{i t f(a)}$ and $\Phi(e^{i s b}) = e^{i s f(b)}$ operator commute for all $s,t\in \mathbb{R}$. We obtain from Remark~\ref{remark exponential of i hermitians} that $f(a)$ and  $f(b)$ operator commute.
	\end{proof}
	
The next lemma shows that piecewise Jordan homomorphisms between the unitary sets of two unital JB$^*$-algebras are unital.  	

\begin{lemma}\label{l every Jordan piecewise homomorphisms unitaries is unital} Let $\Phi: \mathcal{U}(\mathfrak{A}) \rightarrow \mathcal{U}(\mathfrak{B})$ be a piecewise Jordan homomorphism between the unitary sets of two unital JB$^*$-algebras. Then $\Phi (\unit) = \unit.$
\end{lemma}

\begin{proof} Since $\unit$ operator commute with every other unitary element in $\mathcal{U}(\mathfrak{A})$, in particular with itself, it follows from the hypotheses that $\Phi (\unit) = \Phi (\unit^2) = \Phi (\unit)^2$. That is, $\Phi(\unit)$ is a unitary and an idempotent in $\mathfrak{B}$. As seen in Remark~\ref{remark JB*-subalgebra gen by a unitary}, the JB$^*$-subalgebra of $\mathfrak{B}$ generated by $\Phi (\unit)$ is a unital and commutative C$^*$-algebra, so $\Phi (\unit) = \unit.$
\end{proof}

We characterize next piecewise Jordan isomorphisms between the unitary sets of two JBW$^*$-algebras. 

\begin{theorem}\label{teo_piecewise_bij} Let $\mathfrak{A}$ and $\mathfrak{B}$ be two JBW$^*$-algebras, and let $\Phi: \mathcal{U}(\mathfrak{A}) \rightarrow \mathcal{U}(\mathfrak{B})$ be a bicontinuous bijection between the corresponding sets of unitary elements. Then the following conditions are equivalent: 
\begin{enumerate}[$(i)$]
        \item $\Phi$ is a piecewise Jordan isomorphism;
        \item $\Phi(U_u(v)) = U_{\Phi(u)}(\Phi(v)),$ for all $u,v \in \mathcal{U(\mathfrak{A})}$ such that $u$ and $v$ operator commute, and $\Phi^{-1}(U_{u'}(v')) = U_{\Phi^{-1}(u')}(\Phi^{-1}(v')),$ for all $u',v' \in \mathcal{U(\mathfrak{B})}$ such that $u'$ and $v'$ operator commute;
        \item There exists a bijection $f : \mathfrak{A}_{sa} \rightarrow \mathfrak{B}_{sa}$ which is homogeneous, additive on operator commuting elements,  preserves operator commutativity in both directions, and satisfies 
         $$ \Phi(e^{ita}) = e^{itf(a)}, \hbox{ and }  \Phi^{-1} (e^{i t b}) = e^{ i t f^{-1}(b)},$$ for all $a \in \mathfrak{A}_{sa},$ $b \in \mathfrak{B}_{sa},$ and $t \in \mathbb{R}$.
    \end{enumerate}
\end{theorem}

\begin{proof} Lemma~\ref{l every Jordan piecewise homomorphisms unitaries is unital} implies that $\Phi(\unit) = \unit$.\smallskip

$(i)\Rightarrow (ii)$ Take two operator commuting elements $u$ and $v$ in $\mathcal{U} (\mathfrak{A})$. Working with the JBW$^*$-subalgebra of $\mathfrak{A}$ generated by $u, v$ and $\unit,$ which is an associative algebra (see Remark~\ref{remark exponential of i hermitians}), and in particular, a commutative von Neumann algebra, it can be easily deduced that
$$ U_u(v) = u^2\circ v.$$
Therefore, since $\Phi$ is a piecewise Jordan isomorphism we get that 
$ \Phi ( u^2 \circ v) = \Phi(u)^2 \circ \Phi (v)$, where $\Phi(u^2) = \Phi(u)^2$ and $\Phi(u)$ operator commutes with $\Phi(v)$ by hypotheses (cf. Lemma~\ref{lemma_op_comm_unit}). By working on the associative JBW$^*$-subalgebra generated by $\Phi(u), \Phi(v),$ and $\unit$, we derive that $U_{\Phi(u)}(\Phi (v)) = \Phi(u)^2 \circ \Phi(v)$. Therefore, $\Phi (U_u(v)) = U_{\Phi(u)}(\Phi(v))$ as we wanted to prove. Similar arguments applied to $\Phi ^{-1}$ prove the second part of the statement in $(ii)$. \smallskip

$(ii)\Rightarrow (i)$ Let us take two operator commuting elements $u,v \in \mathcal{U}(\mathfrak{A})$. By Lemma~\ref{lemma_op_comm_unit} there exist two operator commuting elements $h$, $k\in \mathfrak{A}_{sa}$ such that $u = e^{i h}, v = e^{i k}$. It follows from the assumptions on $\Phi$ that $\Phi(u^2) =\Phi\left( U_{u} (\unit) \right) =  U_{\Phi\left(u\right)} (\Phi\left(\unit\right)) = U_{\Phi\left(u\right)} (\unit) = \Phi(u)^2$. An induction argument leads to $\Phi(u^n) = \Phi(u)^n$ for every $n \in \mathbb{N}$. We shall next show that $\Phi$ preserves one-parameter unitary groups.\smallskip

The element $w =  e^{i h/2} \in \mathcal{U}(\mathfrak{A})$  operator commutes with $u$ and $u^{-1}$ (see, for example, Lemma~\ref{lemma_op_comm_unit}) and satisfies $w^2  = u$. The conclusion in the previous paragraph implies that $\Phi (w)^2 = \Phi (w^2) = \Phi (u).$  The assumptions on $\Phi$ give $$\unit = \Phi(\unit) = \Phi(U_w(u^{-1})) = U_{\Phi(w)}(\Phi(u^{-1})).$$ The JBW$^*$-subalgebra generated by $\Phi(u^{-1})$, $\Phi(w)$, and $\unit$ is a JB$^*$-subalgebra of a C$^*$-algebra $A$ by the Shirshov-Cohn theorem, and hence by considering the associative product on $A$ (denoted by juxtaposition) we can write 
$$ \unit = U_{\Phi(w)} (\Phi(u^{-1})) = \Phi(w) \Phi(u^{-1}) \Phi(w),$$
which implies that 
$$\Phi(u^{-1}) = \Phi(w)^{-2} = (\Phi(w)^2)^{-1}  = \Phi(u)^{-1}.$$
Therefore $\Phi(u^{-n}) = \Phi((u^{-1})^{n})= \Phi(u^{-1})^n = \Phi(u)^{-n}$, for all natural $n$, which in turn gives  $\Phi(u^k) = \Phi(u)^k,$ for every $k \in \mathbb{Z}.$ Let us now take $n, n' \in \mathbb{Z}$ and $m, m' \in \mathbb{N}$. We deduce from the previous conclusions that 
$$
\begin{aligned}
    \Phi \left ( e^{i(\frac{n}{m} + \frac{n'}{m'}) h} \right) &= \Phi \left ( e^{i \frac{n m' + n' m}{m m' } h} \right) 
    =  \Phi \left ( e^{i \frac{1}{m m'} h} \right)^{n m' + n' m}\\
    &= \Phi\left( e^{i \frac{1}{m m'} h}\right)^{n m'} \circ \Phi \left(e^{i \frac{1}{m m'} h} \right)^{n' m} = \Phi\left( e^{i \frac{n}{m} h }\right) \circ \Phi \left(e^{i \frac{n'}{m'} h} \right).
\end{aligned}
$$
The continuity of $\Phi$ assures that the previous identity holds when $\frac{n}{m} $ and $\frac{n'}{m'}$ are replaced with general real numbers $s, t$, i.e., 
$$ \Phi \left ( e^{i(s + t) h} \right) = \Phi\left( e^{i s h}\right) \circ \Phi \left(e^{i t h} \right) \ \ (s,t\in \mathbb{R}),$$
which implies that the mapping $t \mapsto \Phi \left(e^{i t h} \right)$ is a continuous one-parameter unitary group, and similarly for $k$. By \cite[Theorem 3.1]{CuetoPeralta2023} there exist self-adjoint elements $h',k' \in \mathfrak{B}_{sa}$ such that $\Phi\left( e^{i t h}\right) = e^{i t h'}$ and $\Phi\left( e^{i t k}\right) = e^{i t k'}$, for all $t \in \mathbb{R}$. Now the identity 
$$U_{e^{i t h}} \left( e^{2 s i k} \right) = U_{e^{i s k}} \left(e^{i 2 t h}\right), \quad s, t \in \mathbb{R} $$
transforms by $\Phi$ into 
$$U_{e^{i t h'}} \left(e^{2 s i k'}\right) = U_{e^{i s k'}} \left(e^{i 2 t h'}\right), \quad s, t \in \mathbb{R}.$$ We deduce from Remark~\ref{remark exponential of i hermitians} that the elements $\Phi(e^{ i t h})= e^{i s h'}$ and $\Phi(e^{ i s k}) = e^{i s k'}$ operator commute for all $s,t\in \mathbb{R}$; in particular $\Phi(e^{i h}) = \Phi(u)$ and $\Phi(e^{i k}) = \Phi(v)$ operator commute. This allows us to rewrite the identity from $(ii)$ in the form 
$$\Phi(u^2\circ v) = \Phi(U_u(v)) = U_{\Phi(u)}(\Phi(v)) = \Phi(u)^2 \circ \Phi(v) = \Phi(u^2) \circ \Phi(v).$$
Since in a JBW$^*$-algebra any unitary element can be written as the square of another unitary element (cf. Remark~\ref{remark JB*-subalgebra gen by a unitary}), we obtain that $\Phi$ is a piecewise Jordan homomorphism. Applying the same arguments to $\Phi^{-1}$, we derive that $\Phi^{-1}$ is a piecewise Jordan homomorphism too. \smallskip

$(i) \Rightarrow (iii)$ By Proposition~\ref{1.2} there exists a homogeneous maps $f : \mathfrak{A}_{sa} \rightarrow \mathfrak{B}_{sa}$ and $g : \mathfrak{B}_{sa} \rightarrow \mathfrak{A}_{sa}$ which are additive on operator commuting elements, preserve operator commutativity, and satisfy $\Phi(e^{i t h}) = e^{i t f(h)},$ and $\Phi^{-1} (e^{ i t k}) = e^{i t g(k)},$ for all $h \in \mathfrak{A}_{sa},$ $k \in \mathfrak{B}_{sa}$, and $t \in \mathbb{R}$. It follows that $$ e^{i t h} = \Phi^{-1}(\Phi(e^{i t h})) = \Phi^{-1}(e^{i t f(h)}) = e^{i t g(f(h))}, \hbox{ for all } t\in \mathbb{R}.$$ We therefore conclude that $g(f(h)) = h,$ for all $h\in \mathfrak{A}_{sa}$. We can similarly get $f(g(k)) = k,$ for all $k\in \mathfrak{B}_{sa}$. Thus $f$ and $g$ are bijective maps with $f = g^{-1}$, and the rest is clear. \smallskip

$(iii) \Rightarrow (i)$ Suppose $u$ and $v$ are two operator commuting elements in $\mathcal{U} (\mathfrak{A})$. By Lemma~\ref{lemma_op_comm_unit}, there exists operator commuting elements $h,k\in \mathfrak{A}_{sa}$ satisfying $u = e^{ih}$ and $v = e^{i k}$. The assumptions on $\Phi$ and $f$ guarantee that $f(h)$ and $f(k)$ operator commute, $\Phi (u) = \Phi (e^{i h}) = e^{i f(h)}$ and $\Phi (v) = \Phi (e^{i k}) = e^{i f(k)}$ operator commute, $f(h + k ) = f(h) + f(k)$, and $$ \Phi (u) \circ \Phi (v)  = e^{i f(h)} \circ e^{i f(k)} = e^{i (f(h)+f(k))} = e^{i f(h+k)} = \Phi (e^{i (h+k)}) = \Phi (u \circ v).$$ This shows that $\Phi$ is a piecewise Jordan homomorphism. We can similarly prove that $\Phi^{-1}$ is a piecewise Jordan homomorphism.
\end{proof}

Proving that the mapping $f$ obtained in Theorem~\ref{teo_piecewise_bij}$(iii)$ is continuous will require a deeper argument.\smallskip

Ongoing research on preservers is intensively exploring when linearity of certain maps can be relaxed from the hypotheses. For examples, Theorem 3.1 in \cite{Hamhalter2025} assures that if $\mathfrak{M}$ is a JBW-algebra without type $I_2$ direct summand and $\mathfrak{J}$ is a JB-algebra, every positive homogeneous mapping  $g : \mathfrak{M} \to  \mathfrak{J}$ which is additive on operator commuting elements is a bounded linear mapping.  Observe that every such a mapping $g$ satisfies that $\|g (p)\|\leq \|g(\unit)\|$, and hence it is bounded on the set of projections of $\mathfrak{M}$.\smallskip 

In this paper we shall need a tool asserting that linearity is an automatic property for homogeneous mappings between JBW-algebras which are additive on operator commuting elements as soon as they are bounded on the closed unit ball. The result is actually a consequence of the Mackey--Gleason--Bunce--Wright theorem for JBW$^*$-algebras, which has been recently established in \cite{EscolanoPeraltaVillena2025_mes} after decades of waiting a confirmation of this conjecture. 

\begin{theorem}\label{t corollary of the Mackey-Gleason} Let $\mathfrak{A}$ be a JBW$^*$-algebra without direct summands of type $I_2$, and let $X$ be a real normed space. Suppose $f: \mathfrak{A}_{sa}\to X$ is a homogeneous mapping which is additive on operator commuting elements, that is, $f(a+b) = f(a) + f(b)$ whenever $a$ and $b$ operator commute in $\mathfrak{A}$. Then the following statements are equivalent:
\begin{enumerate}[$(a)$]
\item $f$ is linear and continuous;
\item $f$ is bounded on the closed unit ball of $\mathfrak{A}_{sa}$. 
\end{enumerate} 
\end{theorem}

\begin{proof} The implication  $(a)\Rightarrow (b)$ is clear. Suppose now that $f$ is bounded on the closed unit ball of $\mathfrak{A}_{sa}$, that is, there exists a positive $L$ such that $\|f(a)\|\leq L$ for all $a\in \mathfrak{A}_{sa}$ with $\|a\|\leq 1$. By homogeneity, $\|f(a)\|\leq L \|a\|,$ for all $a\in \mathfrak{A}_{sa}$. Consider the mapping $\mu = f|_{\hbox{Proj}(\mathfrak{A})}: \hbox{Proj}(\mathfrak{A})\to X,$ $\mu (p) = f(p)$.  Since orthogonal projections in $\mathfrak{A}$ operator commute, it follows from the properties of $f$ that the mapping $\mu$ is a finitely additive measure. It is also clear that $\|\mu (p) \|\leq L$ for all $p\in \hbox{Proj}(\mathfrak{A})$. Theorem 6.2 in \cite{EscolanoPeraltaVillena2025_mes} assures the existence of a bounded linear mapping $T: \mathfrak{A}_{sa}\to X$ such that $f(p) = \mu (p) = T(p),$ for all $p\in \hbox{Proj} (\mathfrak{A})$. \smallskip
	
We shall finally prove that $T = f$. Fix $a\in \mathfrak{A}_{sa}$. The JBW$^*$-subalgebra, $\overline{\mathfrak{A}_{a}}^{w*},$ generated by $a$ and $\unit$ is a commutative von Neumann algebra. It follows from our previous discussion on operator commuting elements that every couple of elements in the self-adjoint part of $\overline{\mathfrak{A}_{a}}^{w*}$ operator commute in $\mathfrak{A}$ (cf. \cite[Theorem]{vandeWet2020} and the previous section).  For each positive $\varepsilon$ we can find mutually orthogonal projections $p_1,\ldots, p_m$ in $\overline{\mathfrak{A}_{a}}^{w*}$ and real numbers $\alpha_1,\ldots, \alpha_m$ such that $\displaystyle \left\| a - \sum_{j=1}^{m} \alpha_j p_j\right\| < \frac{\varepsilon}{2 \max\{L,\|T\|\}}$. The above properties assure that  $$\begin{aligned}
f\left( \sum_{j=1}^m \alpha_j p_j\right) = \sum_{j=1}^m & \alpha_j f\left(p_j\right) = \sum_{j=1}^m \alpha_j T\left(p_j\right) = T\left( \sum_{j=1}^m \alpha_j p_j\right),\\ 
\left\| f\left( \sum_{j=1}^m \alpha_j p_j\right) -f (a) \right\| &=  \left\| f\left( \sum_{j=1}^m \alpha_j p_j -a\right) \right\|\leq L \left\| \sum_{j=1}^m \alpha_j p_j - a \right\| <\frac{\varepsilon}{2}, 
\end{aligned}$$ $\displaystyle \left\| T\left( \sum_{j=1}^m \alpha_j p_j\right) - T (a) \right\| <\frac{\varepsilon}{2},$  and thus $\|f (a) -T(a)\| <\varepsilon$. It follows from the arbitrariness of $\varepsilon>0$  that $f(a) = T(a).$ 
\end{proof}

We shall see in Example~\ref{couterexample spin} that the previous Theorem~\ref{t corollary of the Mackey-Gleason} fails for JBW$^*$-algebras of type $I_2$.\smallskip

We continue with a technical lemma essentially borrowed from \cite{Hamhalter2023}. Henceforth, the symbol $\mathbb{T}$ will stand for the unit sphere of $\mathbb{C}$.

\begin{lemma}\label{inequality} Let $u$ be a unitary element in a unital JBW$^*$-algebra $\mathfrak{A}$, and let $n$ be a natural number  such that $n \nor{u-\unit}<2.$ Then $ n\nor{u-\unit}\le \frac{\pi}{2} \nor{u^n-\unit}.$ 
  \end{lemma}

\begin{proof} It is established in the proof of \cite[Theorem 5.2, page 13]{Hamhalter2023} that for $z\in \mathbb{T}$ and $n\in \mathbb{N}$, the condition $n |z-1|<2$ implies that $n |z-1|\le \frac{\pi}{2} |z^n-1|.$ Since, we can identify the JB$^*$-subalgebra of $\mathfrak{A}$ generated by $u$ and $\unit$ with a unital and commutative von Neumann algebra of the form $C(\Omega),$ for some compact Hausdorff space $\Omega$, a simple combination of this identification with the above conclusion for elements in $\mathbb{T}$ gives $$n\nor{u-\unit}\le \frac{\pi}2\nor{u^n-\unit}.$$ \end{proof}

We prove next that every continuous unital piecewise Jordan homomorphism between sets of unitary elements is locally Lipschitzian at the unit element.

\begin{lemma}\label{lemma:local_lip_prop}
Let $\Phi : \mathcal{U}(\mathfrak{A}) \rightarrow \mathcal{U}(\mathfrak{B})$ be a continuous unital piecewise Jordan homomorphism between the unitary sets of two JBW$^*$-algebras. Then $\Phi $ is locally Lipschitzian at $\unit$, that is, there exist $r, L> 0 $ satisfying the following property: 
$$\hbox{For all } u\in \mathcal{U}(\mathfrak{A}) \hbox{ with } \| u - \unit \| < r, \hbox{ we have } \| \Phi(u) - \unit \| \leqslant L \| u- \unit\|.
$$
\end{lemma}

\begin{proof} Arguing by contradiction we assume the existence of  a sequence $(u_k)_k\subseteq \mathcal{U} (\mathfrak{A})$ satisfying:  
    \[ \nor{u_k-\unit}<\frac 1k, \mbox{ and }  1\geq \e_k= \nor{\Phi(u_k)-\unit}>k \nor{u_k-\unit}.\] The continuity of $\Phi$  implies that $\e_k\to 0$. Let $(m_k)_k$ be a sequence of natural numbers such that
       \[ \frac 1{m_k+1}<\e_k\le \frac 1{m_k}\,.                  \] 
       
The inequality 
        \[  \nor{u_k^{m_k}-\unit}\le \nor{u_k-\unit} \ \nor{u_k^{m_k-\unit}+\cdots +\unit}\le m_k \nor{u_k-\unit}\le \frac{m_k\e_k}k\le \frac 1k,                          \] assures that $u_k^{m_k}\to \unit$ in norm, and so $\Phi(u_k)^{m_k+1}\to \unit$. \\
      
Since we can clearly assume that $m_k>2$ for all $k,$ we deduce that 
\[  2(m_k+1) \nor{\Phi(u_k)-\unit}= 2(m_k+1)\e_k   \le 2 \left(1 + \frac 1{m_k}\right) <3,            \]     
and so $(m_k+1) \nor{\Phi(u_k)-\unit}<2\,.$	Now, Lemma~\ref{inequality} assures that 
             \[  1 = \frac 1{\e_k} \nor{\Phi(u_k)-\unit}\le (m_k+1) \nor{\Phi(u_k)-1} \le \frac \pi 2\nor{\Phi(u_k)^{m_k+1}-\unit},         \]
           which contradicts that $\Phi(u_k)^{m_k+1}\to \unit$. 
\end{proof}

We can now establish a precise description of all bicontinuous piecewise Jordan isomorphisms between unitary sets of JBW$^*$-algebras.  

\begin{theorem}\label{theo: 4.3} Let $\mathfrak{A}$ and $\mathfrak{B}$ be JBW$^*$-algebras without direct summands of type $I_1$ and $I_2$. Let $\Phi : \mathcal{U}(\mathfrak{A}) \rightarrow \mathcal{U}(\mathfrak{B})$ be a bijective and bicontinuous piecewise Jordan isomorphism. Then there exist a linear Jordan $^*$-isomorphism $\theta: \mathfrak{A} \rightarrow \mathfrak{B}$,  a real linear mapping $\beta: \mathfrak{A_{sa}}\rightarrow Z(\mathfrak{B}_{sa})$, and an invertible central element $c \in \mathfrak{B}_{sa}$ such that $$ \Phi(e^{i a}) = e^{i \beta (a)}\circ e^{i  c\circ\theta(a)} = e^{i \beta(a)} \circ \theta \left( e^{i  \theta^{-1}( c )\circ a}\right),$$  for all $a\in \mathfrak{A}_{sa}$. In particular, the existence of such a piecewise isomorphism  $\Phi: \mathcal{U}(\mathfrak{A})\to \mathcal{U}(\mathfrak{B})$ assures that $\mathfrak{A}$ and $\mathfrak{B}$ are Jordan $^*$-isomorphic.
\end{theorem}

\begin{proof} We know that $\Phi (\unit) = \unit$ (cf. Lemma~\ref{l every Jordan piecewise homomorphisms unitaries is unital}). \smallskip  
 
Lemma~\ref{lemma:local_lip_prop} assures that ${\Phi}$ enjoys the local Lipschitz property at $\unit$ for some positive constant $r, L> 0 $. By Theorem~\ref{teo_piecewise_bij} we also know the existence of a bijection $f: \mathfrak{A}_{sa} \rightarrow \mathfrak{B}_{sa}$ which is homogeneous, additive on operator commuting elements,  preserves operator commutativity in both directions, and satisfies 
$${\Phi}\left (e^{ita} \right) = e^{itf(a)}, \hbox{ for all } t \in \mathbb{R}, a \in \mathfrak{A}_{sa}.$$

Our next goal will consist in proving that $f$ is bounded and linear by applying Theorem~\ref{t corollary of the Mackey-Gleason}. For this purpose, we need to show the boundedness of $f$ on the unit ball of $\mathfrak{A}_{sa}$. Fix $a\in \mathfrak{A}_{sa}$. By employing the local Lipschitz property from Lemma~\ref{lemma:local_lip_prop}, for $t\in \mathbb{R}$ with $|t|$ small enough we have        
\[  \nor{ e^{itf(a)}-\unit} = \nor{{\Phi}(e^{iat})-\unit}\le L\nor{e^{iat}-\unit},    \]   and 
\[  \biggl\| \frac{e^{itf(a)}-\unit }{t}  \biggr\|  \le L \left\| \frac{e^{i t a}-\unit}t\right\|, \hbox{ for all real } t \hbox{ with } |t| \hbox{ small}.\] Taking limits at $t\to 0$ we are led to $\nor{f(a)}\le L\ \|a\|.$ Theorem~\ref{t corollary of the Mackey-Gleason} implies that $f$ is a bounded linear mapping. We can similarly prove that $f^{-1}$ is also continuous. \smallskip

Now, by \cite[Theorem 8.10]{EscolanoPeraltaVillena2025_OpCommut} there exist an invertible element $c \in Z(\mathfrak{B})_{sa}$, a Jordan isomorphism $\theta: \mathfrak{A}_{sa} \rightarrow \mathfrak{B}_{sa},$ and a linear mapping $\beta: \mathfrak{A}_{sa} \rightarrow Z(\mathfrak{B}_{sa})$ satisfying 
$$ f(a) = c \circ \theta(a) + \beta(a)$$
for all $a \in \mathfrak{A}_{sa}$. Clearly $\theta$ extends to a linear Jordan $^*$-isomorphism from $\mathfrak{A}$ onto $\mathfrak{B}$. Moreover, having in mind that $\beta (a)$ is a central element, we conclude that 
$$\begin{aligned}
{\Phi}(e^{i a}) & =  e^{i f(a)} =e^{i (c\circ \theta(a) + \beta(a))}= e^{i \beta(a)}\circ e^{i c\circ \theta(a)}=   e^{i \beta(a)} \circ \theta \left( e^{i  \theta^{-1}( c)\circ a}\right),	
\end{aligned} $$ for all $a\in \mathfrak{A}_{sa}$,
which completes the proof.
\end{proof}

In the case of JBW$^*$-factors not of type $I_1$ or $I_2$ we have the following corollary.

\begin{corollary}\label{c factors not of type I1 or I2} Let $\mathfrak{A}$ and $\mathfrak{B}$ be JBW$^*$-algebra factors not of type $I_1$ or $I_2,$ and let $\Phi : \mathcal{U}(\mathfrak{A}) \rightarrow \mathcal{U}(\mathfrak{B})$ be a mapping satisfying the hypotheses in the previous theorem. Then there exists a Jordan $^*$-isomorphism $\theta: \mathfrak{A} \rightarrow \mathfrak{B},$  a real linear functional $\beta$ on $\mathfrak{A}_{sa}$, and a non-zero real number $\alpha$ such that $$ \Phi( e^{i a}) = e^{i \beta(a)} e^{i \alpha \theta (a)}= e^{i \beta(a)}  \theta \left( e^{i \alpha a}\right) = e^{i \beta(a)}  \theta \left( (e^{i a})^\alpha\right) ,$$ for all $a\in \mathfrak{A}_{sa}$. 
\end{corollary}

We continue with a couple of technical results. 

\begin{lemma}\label{l Phi preserves central elements} Let $\mathfrak{A}$ and $\mathfrak{B}$ be unital JB$^*$-algebras, and let $\Phi : \mathcal{U}(\mathfrak{A}) \rightarrow \mathcal{U}(\mathfrak{B})$ be a bijective piecewise Jordan isomorphism. Then the following statements hold:
\begin{enumerate}[$(a)$]
\item $\Phi \left(Z(\mathfrak{A})\cap \mathcal{U}(\mathfrak{A}) \right) = Z(\mathfrak{B})\cap \mathcal{U}(\mathfrak{B})$;
\item $\Phi \left(\hbox{Symm}(\mathfrak{A}) \right) = \hbox{Symm}(\mathfrak{B})$, and there exists a bijective mapping preserving operator commutativity $\Psi: \hbox{Proj} (\mathfrak{A}) \to \hbox{Proj} (\mathfrak{B})$  given by $\Phi (\unit- 2p) = \unit - 2 \Psi (p)$ for all $p\in  \hbox{Proj} (\mathfrak{A})\backslash\{0,\unit\}$, $\Psi (\unit)= \unit$ and $\Psi (0)=0$;
\item $\Phi \left(Z(\mathfrak{A})\cap \hbox{Symm}(\mathfrak{A}) \right) = Z(\mathfrak{B})\cap \hbox{Symm}(\mathfrak{B})$;
\item $\mathfrak{A}$ is a factor if, and only if, $\mathfrak{B}$ is a factor. 
\end{enumerate}
\end{lemma}

\begin{proof} $(a)$ Suppose $u\in Z(\mathfrak{A})\cap \mathcal{U}(\mathfrak{A})$. It follows from the hypotheses that $\Phi (u)$ operator commute with every element in $\Phi \left( \mathcal{U}(\mathfrak{A}) \right) = \mathcal{U}(\mathfrak{B})$. Since every element in $\mathfrak{B}$ can be written as a linear combination of four unitaries in $\mathfrak{B}$, $\Phi (u)$ must be a central element, that is, $\Phi (u)\in Z(\mathfrak{B})\cap \mathcal{U}(\mathfrak{B}).$ A similar argument applied to $\Phi^{-1}$ gives the desired identity.\smallskip
	
$(b)$ Take $u \in \hbox{Symm}(\mathfrak{A}) = \mathcal{U} (\mathfrak{A})\cap \mathfrak{A}_{sa}$. By the properties of $\Phi$ and Lemma~\ref{l every Jordan piecewise homomorphisms unitaries is unital} we also have $\Phi (u)^2 = \Phi(u^2) = \Phi (\unit) = \unit$. It follows that $\Phi (u)^* = \Phi (u)$ is a symmetry (cf. Remark~\ref{remark JB*-subalgebra gen by a unitary}). Since every symmetry in a JBW$^*$-algebra writes in the form $\unit-2 p$ for a unique projection $p$, the statement follows from the previous conclusion and the hypotheses on $\Phi$.\smallskip

$(c)$ Is a consequence of $(a)$ and $(b)$.\smallskip

$(d)$ If $\mathfrak{A}$ is a factor, $Z(\mathfrak{A})\cap \hbox{Symm}(\mathfrak{A}) = \{\pm \unit\}$. It follows from $(d)$ (see also Lemma~\ref{l every Jordan piecewise homomorphisms unitaries is unital}) that $Z(\mathfrak{B})\cap \hbox{Symm}(\mathfrak{B}) = \{\pm \unit\}$, and thus $\mathfrak{B}$ is a factor too. 
\end{proof}

The optimal conclusion concerning bicontinuous piecewise isomorphisms between unitary sets of JBW$^*$-factors is stated in the next result, which also improves the conclusion for von Neumann factors in \cite[Theorem 5.4]{Hamhalter2023}.

\begin{theorem}\label{T2} Let $\mathfrak{A}$ be a JBW$^*$-factor not of type $I_2$, and let $\mathfrak{B}$ be a JBW$^*$-algebra. Let
 $\Phi: \mathcal{U}(\mathfrak{A})\to \mathcal{U}(\mathfrak{B})$	be a bicontinuous bijection. Then $\Phi$ is a piecewise Jordan isomorphism if, and only if, there is a Jordan $^*$-isomorphism $\theta: \mathfrak{A}\to \mathfrak{B}$ such that one of the following conditions holds for all $u\in \mathcal{U}(\mathfrak{A})$: 
 \begin{enumerate}[$(a)$]
 \item  $\Phi(u)=\theta(u)$, for all $u \in \mathcal{U}(\mathfrak{A})$.
 \item  $\Phi(u)=\theta{(u^{-1})}$, for all $u \in \mathcal{U}(\mathfrak{A})$.
 \end{enumerate}
 \end{theorem}  
 
\begin{proof} Lemma~\ref{l Phi preserves central elements} assures that $\mathfrak{B}$ is a JBW$^*$-factor. If $\mathfrak{A}$ if of type $I_1$, we have $\mathcal{U} (\mathfrak{A}) = \mathbb{T}.$ Since $\mathcal{U} (\mathfrak{A})\subseteq Z(\mathfrak{A})$, Lemma~\ref{l Phi preserves central elements} guarantees that $$\mathcal{U} (\mathfrak{B}) = \Phi \left( \mathcal{U} (\mathfrak{A})\cap Z(\mathfrak{A})  \right) \subseteq Z(\mathfrak{B}).$$ Since the linear span of all unitaries in  $\mathfrak{B}$ is the whole algebra, it follows that $\mathfrak{B}$ is associative. Therefore $\mathfrak{B} =\mathbb{C}$. Note that, by hypotheses, the mapping $\Phi: \mathbb{T} \unit= \mathcal{U}(\mathfrak{A}) \to \mathcal{U}(\mathfrak{B})= \mathbb{T} \unit$ is a bicontinuous group isomorphism (note that $\mathfrak{A}$ is associative). So, $\Phi$ is the identity or the conjugation on $\mathbb{T}$, and hence $(a)$ or $(b)$ holds \cite[(23.31)]{HewRossBookVol1}.\smallskip

We assume next that $\mathfrak{A}$ is a JBW$^*$-factor not of type $I_1$ or $I_2$. It follows from the conclusion in the previous paragraph that $\mathfrak{B}$ is a JBW$^*$-factor not of type $I_1$.  Let us prove now that $\mathfrak{B}$ is not a JBW$^*$-factor of type $I_2$. Namely, a factor of type $I_2$ (also called a spin factor), $\mathfrak{S}$, only contains minimal projections besides the trivial ones. Trivially $0$ and $\unit$ operator commute with any other projection, however if $p$ and $q$ are two operator commuting minimal projections in $\mathfrak{S}$ it follows that $$\lambda p =  U_p (q) = p\circ q = U_q (p) = \mu q,$$ for some $\lambda,\mu \in \mathbb{C}$, and thus $p=q$ or $p q =0$. That is, the set of all projections operator commuting with a non-trivial (minimal) projection $p$ in $\mathfrak{S}$ reduces to $\{0,\unit, p, \unit-p \}$. By Lemma~\ref{l Phi preserves central elements}$(b)$, we can find a bijection $\Psi: \hbox{Proj}(\mathfrak{A})\to \hbox{Proj}(\mathfrak{B})$ preserving operator commutativity with $\Psi (\unit)= \unit$ and $\Psi (0)=0$. Since $\mathfrak{A}$ is a JBW$^*$-factor not of type $I_1$ or $I_2$, we can find at least three non-zero mutually orthogonal projections $p_1,p_2,p_3$. Noting that the set of projections in $\mathfrak{A}$ operator commuting with $p_1$ contains $\{0,\unit, p_2, p_3, \unit-p_1 \}$, it follows that $\Psi (0)=0,$ $\Psi (\unit)=\unit$, $\Psi (p_2),$ $\Psi (p_3),$ and $\Psi (\unit-p_1)$ operator commute with the non-trivial projection $\Psi (p_1)$. But this is impossible in a JBW$^*$-factor of type $I_2$.\smallskip 

We can assume next that $\mathfrak{A}$ and $\mathfrak{B}$ are JBW$^*$-factors not of type $I_1$ or $I_2$. By Corollary~\ref{c factors not of type I1 or I2}, there exists a Jordan $^*$-isomorphism $\theta: \mathfrak{A} \rightarrow \mathfrak{B},$  a real linear functional $\beta$ on $\mathfrak{A}_{sa}$, and a non-zero real number $\alpha$ such that \begin{equation}\label{eq expression of phi in Thm 3.10}  \Phi( e^{i a}) = e^{i \beta(a)} e^{i \alpha \theta (a)}= e^{i \beta(a)}  \theta \left( e^{i \alpha a}\right),
\end{equation} for all $a\in \mathfrak{A}_{sa}$.\smallskip

Let us take a non-trivial projection $p\in \mathfrak{A}$. Then $\theta(p)$ is also nontrivial projection. The identity in \eqref{eq expression of phi in Thm 3.10} leads to 
$$\begin{aligned}
\Phi(\unit - 2 p ) &= \Phi(\unit - p + e^{i\pi} p) = \Phi(e^{i\pi p})= e^{i\pi\beta(p)}  e^{i \alpha \pi \theta(p)} \\
&=  e^{i\pi \beta(p)} [\unit-\theta(p) + e^{i \alpha \pi} \theta(p)].
\end{aligned}$$
Set $\lambda=e^{i\pi \beta(p)}$. As $\Phi$ preserves symmetries (see Lemma~\ref{l Phi preserves central elements}), we infer that $$ \lambda^2 (\unit-\theta(p))= \unit-\theta(p),\hbox{ and } \lambda^2 e^{i 2 \alpha \pi} \theta (p) = \theta (p).$$
Therefore, $\lambda=\pm 1$, which implies that $\beta(p)$ is an integer. Furthermore, the element $\unit-\theta(p)+ e^{i\alpha \pi} \theta(p)$ is a symmetry as well, and thus $ e^{i\alpha \pi}= \pm 1,$
and so $\alpha$ must be an integer. The case $e^{i\alpha \pi}=1$ is impossible, since in such a case $\Phi (\unit - 2 p ) = \pm \unit = \pm \Phi (\unit) =  \Phi (\pm \unit)$, which contradicts that $\Phi$ is a bijection. So, $e^{i\alpha \pi}=-1,$ that is, $\alpha\in 1+2 \mathbb{Z}$. We have therefore  shown that \begin{equation}\label{eq imgae of symmetries} \left\{\begin{aligned} \Phi (s) &\in \{\pm \theta (s)\} = \{\theta (\pm  s)\}, \hbox{ for all }  s\in \hbox{Symm}(\mathfrak{A}),
		\\ \alpha&\in 1+2 \mathbb{Z}, \hbox{ and } \beta (p)\in \mathbb{Z} \hbox{ for every projection } p\in \mathfrak{A}.\end{aligned}\right.
\end{equation}

Take two non-zero operator commuting projections $p,q\in \mathfrak{A}$ which are exchanged by a symmetry, that is, there exists a symmetry $s\in \mathfrak{A}$ satisfying $U_{s} (p ) =q$ (usually denoted by $p\sim_1 q$ \cite[5.1.4]{HOS}). Note that $\Phi (s) = \sigma \theta (s)$ for some $\sigma\in \{\pm1\}$ (cf. \eqref{eq imgae of symmetries}). Clearly, $e^{i p}$ and $e^{iq}$ operator commute, and by \eqref{eq expression of phi in Thm 3.10} and Theorem~\ref{teo_piecewise_bij} we get $$\begin{aligned}
e^{i \beta(q)} \theta\left( e^{i \alpha q} \right) &= \Phi \left(e^{i q}\right)= \Phi \left(U_s \left(e^{i p}\right)\right)= U_{\Phi(s)} \left(\Phi(e^{i p})\right) \\
&= U_{\sigma \theta(s)} \left(e^{i \beta(p)} \theta\left( e^{i \alpha p} \right) \right) =e^{i \beta(p)}  U_{ \theta(s)} \left(  \theta\left(  e^{i \alpha p}\right) \right) \\
&= e^{i \beta(p)}   \theta\left(  U_{ s} \left(  e^{i \alpha p}\right) \right) = e^{i \beta(p)} \theta\left( e^{i \alpha q} \right), 
\end{aligned}$$ which, in particular, implies that $e^{i \beta(q)} = e^{i \beta(p)},$ equivalently, $\beta(q)-\beta(p)\in 2\pi \mathbb{Z}$. However, according to what we have seen above, $\beta$ takes integer values on non-zero projections, and thus $\beta(q)=\beta(p)$. We have established that for every pair of non-zero operator commuting projections $p,q$ in $\mathfrak{A}$ we have \begin{equation}\label{eq beta is constant on equivalent projections} p\sim_1 q \Rightarrow   \beta(q) = \beta(p). 
\end{equation}

Combining the facts that $\mathfrak{A}$ and $\mathfrak{B}$ are JBW$^*$-factors with Lemma~\ref{l Phi preserves central elements}$(a)$ we arrive to the conclusion that $\Phi|_{Z(\mathfrak{A})\cap \mathcal{U} (\mathfrak{A})}: Z(\mathfrak{A})\cap \mathcal{U} (\mathfrak{A})\to Z(\mathfrak{B})\cap \mathcal{U} (\mathfrak{B})$ induces continuous automorphism on $\mathbb{T}$, which by \eqref{eq expression of phi in Thm 3.10} is given by $$\Phi \left(e^{i t} \unit \right)= \Phi \left(e^{i t \unit} \right) = e^{i t (\beta (\unit)+\alpha)} \unit.$$ As we commented above, the only continuous automorphisms of the circle group are the transformations $z\to z$, $z\to z^{-1} = \overline{z}$, we conclude that precisely one of the following identities takes place:
\begin{enumerate}[$(a)$]
\item $\beta (\unit)+\alpha=1$;
\item $\beta (\unit)+\alpha =-1$.
\end{enumerate}

Our next goal is to show that $\beta =0$. Suppose first that $\mathfrak{A}$ is a JBW$^*$-factor of type $I_n$ with $2<n<\infty$. Every projection in $\mathfrak{A}$ can be written as the sum of at most $n$ mutually orthogonal minimal projections in $\mathfrak{A}$. Moreover, any two minimal projections in $\mathfrak{A}$ are exchanged by a symmetry \cite[Lemma 5.3.2$(ii)$]{HOS}. Clearly, $\beta\neq 0$ if and only if $\beta (p)\neq 0$ for some minimal projection $p\in \mathfrak{A}$. It follows from \eqref{eq beta is constant on equivalent projections} that $\beta$ takes a constant and non-zero entire value at all minimal projections in $\mathfrak{A}$. Namely, since $n>2$, given two minimal projections $e,f$ in $\mathfrak{A}$, we can find a third (minimal) projection $g$ orthogonal to both $e,f$. Then $\beta(e)=\beta(g)=\beta(f)$.\smallskip

By structure theory, we can find mutually orthogonal minimal projections $p_1,$ $\ldots,$ $ p_n$ in $\mathfrak{A}$ \cite[5.3.3]{HOS}. If $\mathbb{Z}\ni \alpha \neq \pm 1$, we can find at least two different unitary complex numbers $\lambda_1=e^{it}$ and $\lambda_2=e^{is}$, $s\neq t\in \mathbb{R}$, such that $\lambda_j^\alpha=1$ for all $j =1,2$. Consider the following (different) unitaries
$$
\begin{aligned}
	u_1&= \lambda_1 p_1+\lambda_2 p_2 +\unit -p_1-p_2 = e^{ia}\\
	u_2&= \lambda_2 p_1+\lambda_1  p_2 +\unit-p_1-p_2 = e^{ib}\,
\end{aligned}
$$ with $a= t p_1 + s p_2,$ and $b= s p_1 + t p_2.$ By construction, $u_1^\alpha =u_2^\alpha = \unit,$ and $$ \beta(a)=\beta(b)= (s+t) \beta(p_1) + (n-2) \beta(p_1).$$   Therefore $\Phi(u_1)= e^{i \beta(a)} \theta(u_1^{\alpha}) =  e^{i \beta(b)} \theta(u_2^{\alpha}) =\Phi(u_2) $, which contradicts the fact that $\Phi$ is injective. This proves that $\alpha\in \{\pm 1\}$. However, $\{\pm 1\} \ni \alpha+\beta (\unit) = \alpha + n \beta (p_1)$, equivalently, $n \beta (p_1) \in \{-2,0,2\}$, which is impossible because $n \beta (p_1)\in n (\mathbb{Z}\backslash\{0\})$ with $2<n$. Therefore $\beta =0,$ $\alpha\in \{\pm1\}$ and we conclude the proof in this case.\smallskip

If $\mathfrak{A}$ is a JBW$^*$-factor of type $I_{\infty}$, every projection in $\mathfrak{A}$ is the weak$^*$-limit of the form $\sum_{i\in \Gamma} p_i$, where the $p_i$'s are mutually orthogonal minimal projections in $\mathfrak{A}$. We cannot argue as before because the functional $\beta$ is not necessarily weak$^*$-continuous. However, by the density of the linear span of all projections in $\mathfrak{A}$, if $\beta$ is not zero, there exists a projection $p\in \mathfrak{A}$ such that $\beta (p)\neq 0$. If $p$ decomposes as a finite sum of mutually orthogonal minimal projections we can argue as in the previous paragraph to get a contradiction, otherwise, $p =w^*\hbox{-}\sum_{i\in \Gamma} p_i$ where $\Gamma$ is infinite. By the axiom of choice we can write $\Gamma$ as the disjoint union of two subsets $\Gamma_1$ and $\Gamma_2$ such that they all have the same cardinality as $\Gamma$. By defining $p_k = w^*\hbox{-}\sum_{i\in \Gamma_k} p_i$ with $k=1,2$, we find two mutually orthogonal projections in $\mathfrak{A}$ which, according to \cite[Lemma 5.3.2 and Lemma 5.2.9]{HOS}, satisfy $p\sim_1 p_k$ for all $k=1,2$. We deduce from \eqref{eq beta is constant on equivalent projections} that $\beta (p_1)+ \beta (p_2) = \beta (p_1+p_2) =\beta (p) = \beta (p_1)= \beta (p_2)$,  which is impossible. Therefore, $\beta =0$ as we desired.\smallskip

We finally assume that $\mathfrak{A}$ is a JBW$^*$-factor not of type $I$.  Suppose, as before, that $\beta (p)\neq 0$ for some projection $p\in \mathfrak{A}$. Since an abelian projection in $U_p (\mathfrak{A})$ is abelian in $\mathfrak{A}$, the JBW$^*$-algebra  $U_p (\mathfrak{A})$ contains no direct summand of type $I$, hence the halving lemma proves the existence of two projections $p_1,p_2 \in \mathfrak{A}$ with $p_1\sim_1 p_2$ and $p_1 + p_2 = p$. An induction argument assures that for each natural number $m$ there exist mutually orthogonal projections $p_1,\ldots, p_m$ in $\mathfrak{A}$ with $p_1+\ldots+ p_m=p$ and $p_i\sim_1 p_j$ for all $i,j\in \{1,\ldots, m\}$. The conclusion in \eqref{eq beta is constant on equivalent projections} implies that $\beta(p_i) = \beta(p_j)\in \mathbb{Z}$ for all $i,j\in \{1,\ldots, m\}$. Therefore, $$0\neq \beta (p) = \sum_{j=1}^{m} \beta (p_j) = m \beta(p_1)\in m\mathbb{Z},$$ which is impossible. Then $\beta=0$, and this completes the proof.  
\end{proof}     

It is perhaps timely to conclude this section with some comments on the connections between piecewise Jordan isomorphisms between two sets of symmetries of two unital JB$^*$-algebras and the bijective transformations  preserving commutativity and symmetric difference between their posets of projections in the line started in Lemma~\ref{l Phi preserves central elements}$(b)$.\smallskip

Let $\mathfrak{A}$ be a unital JB$^*$-algebra whose lattice of projections is denoted by $\hbox{Proj}(\mathfrak{A})$. We have already commented the existence of a bijective correspondence between the sets  $\hbox{Proj}(\mathfrak{A})$ and $\hbox{Symm}(\mathfrak{A})$  given by $p\stackrel{\Upsilon}{\mapsto} \unit-2p$. The \emph{symmetric difference} of $p,q$ in $\hbox{Proj}(\mathfrak{A})$ is defined as  $p\Delta q := p+q-2p\circ q.$ If $p$ and $q$ operator commute, the symmetric difference of $p$ and $q$ is a projection in $\mathfrak{A}$ (actually, $p\Delta q\in \hbox{Proj}(\mathfrak{A})$ if, and only if, $p$ and $q$ operator commute). Furthermore, it is easy to check that, in such a case, $\Upsilon(p)$ and $\Upsilon(q)$ operator commute and 
\[  \Upsilon(p\Delta q)= \Upsilon(p)\circ\Upsilon(q)\,. \] Furthermore, if $\mathfrak{B}$ is another unital JB$^*$-algebra, and $\Phi: \hbox{Symm}(\mathfrak{A})\to \hbox{Symm}(\mathfrak{B})$ is a bijection, it follows that $\Phi$ is a piecewise Jordan isomorphisms if, and only if, the mapping $\Psi_\Phi :\hbox{Proj}(\mathfrak{A})\to \hbox{Proj}(\mathfrak{B}),$ $\Psi_\Phi (p) = \frac12 \left( \unit -\Phi (\unit-2 p) \right)$ is a bijection preserving operator commutativity and symmetric differences of operator commuting projections in both directions.\smallskip

In the setting of von Neumann algebras, examples of bijections preserving commutativity and symmetric differences between projection lattices can be provided by the so-called {projection orthoisomorphisms}. We recall that a \emph{projection orthoisomorphism} between von Neumann algebras $A$ and $B$ is a bijective correspondence between their projection lattices which preserves orthogonality (see \cite{Dye}). A celebrated result by Dye prove that if $A$ and $B$ are von Neumann factors not of type $I_{2n}$ and $\Phi :\mathcal{U} (A)\to \mathcal{U} (B)$ is a group isomorphism (and hence $\Phi(\hbox{Symm}({A})) = \hbox{Symm}({B})$), the associated mapping $\Psi_{\Phi|_{\hbox{Symm}({A})}} :\hbox{Proj}({A})\to \hbox{Proj}({B}),$ preserves commutativity and the symmetric differences of commuting projections \cite[Lemma 9]{Dye}, and it is further a projection orthoisomorphism \cite[Lemma 13]{Dye}. It is quite natural to ask up to which point a similar result holds for JBW$^*$-factors. The handicap in the latter setting is that the Jordan product of two unitaries is not, in general, a unitary.

\section{Quadratic maps}\label{sec: 4 quadratic maps}

Following \cite{FriHak88}, we shall say that an arbitrary set $G$ is a \emph{quadratic semigroup} if there is a mapping $G\times G\to G$, $(a,b)\mapsto Q(a)(b)$ satisfying $$Q(Q(a)(b)) (c) = Q(a) Q(b) Q(a) (c), \hbox{ for all } a,b,c\in G,$$ where for each $x\in G$, $Q(x)$ stands for the mapping obtained from the previous mapping on $G\times G$ by fixing the element $a$ as $x$. If $A$ is a C$^*$-algebra and we define $Q(a)(b) = a b a$ (or $Q(a)(b) = a b^* a$) it is easy to see that $A$ is a quadratic semigroup for this mapping. If $\mathfrak{A}$ is a JB$^*$-algebra, by the fundamental identity of Jordan algebras (see \cite[2.4.18]{HOS}) we have $$U_{a} U_{b} U_{a} (c) = U_{U_a (b)} (c) \ \ (a,b,c\in \mathfrak{A}),$$ and so, by defining $Q(a) (b) = U_a (b)$ or $Q(a) (b) = U_a (b^*),$ we induce a structure of quadratic semigroup on $\mathfrak{A}$. This can be further extended to the wider setting of JB$^*$-triples (cf. \cite{FriHak88}). The advantage of employing this structure of quadratic semigroup on a unital JB$^*$-algebra $\mathfrak{A}$ is that, though $\mathcal{U} (\mathfrak{A})\circ  \mathcal{U} (\mathfrak{A})\nsubseteq \mathcal{U} (\mathfrak{A}),$ we still have $U_u (v), U_u (v^*)\in \mathcal{U} (\mathfrak{A}),$ for all $u,v\in \mathcal{U} (\mathfrak{A})$ \cite[Lemma 2.1]{CuetoPeralta2023}. So, $ \mathcal{U} (\mathfrak{A})$ is a \emph{quadratic semigroup} for the quadratic products defined above.\smallskip 

According to \cite[Definition 2.1]{FriHak88}, a mapping $\Phi$ between two JB$^*$-algebras $\mathfrak{A}$ and $\mathfrak{B}$ is called a \emph{quadratic map} if the identity \begin{equation}\label{eq def of quadritc maps} \Phi(\{a,b,a\})\!=\Phi\left(U_a(b^*)\right) =U_{\Phi(a)}(\Phi(b)^*))= \{\Phi(a),\Phi(b),\Phi(a)\},
\end{equation} holds for all $a, b \in \mathfrak{A}$. We can similarly define quadratic maps between JB-algebras as those maps preserving quadratic products of the form $U_a (b)$. Observe that quadratic maps are essentially determined by the triple product.\smallskip

A milestone achievement by Friedman and Hakeda (cf. \cite[Proposition 2.14, Theorem 2.15]{FriHak88}) proves that every quadratic bijection between two JBW$^*$-algebras (in particular, between two von Neumann algebras) $\mathfrak{A}$ and $\mathfrak{B}$, where $\mathfrak{A}$ admits no direct summands of type $I_1$, is a real linear isometry. The result is actually deeper, and an appropriate conclusion holds when $\mathfrak{A}$ and $\mathfrak{B}$ are JBW$^*$-triples with no ``abelian'' direct summands.\smallskip

Let $\mathfrak{A}$ and $\mathfrak{B}$ be JB$^*$-algebras. We shall say that a mapping $\Phi: \mathfrak{A}\to \mathfrak{B}$, or from $\mathcal{U}(\mathfrak{A})$ to $\mathcal{U}(\mathfrak{B}),$ is a \emph{quadratic mapping on operator commuting elements} if the identity in \eqref{eq def of quadritc maps} holds only when $a$ and $b$ operator commute in $\mathfrak{A}$. Observe that every $a$ in $\mathfrak{A}_{sa}$ operator commutes with itself, and hence $\Phi\{a,a,a\} = \{\Phi(a),\Phi(a),\Phi (a)\}$, for all $a\in \mathfrak{A}_{sa}$. Quadratic maps on operator commuting elements between JB-algebras (in particular between self-adjoint parts of C$^*$-algebras) are similarly defined.\smallskip

We are naturally led to the questions whether, in the result by Friedman and Hakeda commented above, we can weaken the hypothesis on $\Phi$ to be a quadratic mapping on operator commuting elements, and if we can study maps between JB-algebras.\smallskip

This is the unique point of this paper where we employ the structure of JB$^*$-triple underlying each JB$^*$-algebra. Recall that a JB$^*$-triple is a complex Banach space $E$ together with a (continuous) triple product $\{.,.,.\}: E^3 \to E$, which is symmetric and bilinear in the outer variables and conjugate-linear in the middle one, and satisfies the following axioms:
\begin{enumerate}[$(1)$]
	\item $L(x,y)L(a,b) = L(L(x,y)(a),b) $ $- L(a,L(y,x)(b))$ $+ L(a,b) L(x,y),$ for all $a,$ $b,$ $x,$ $y\in E$, where for $a,b\in E$, the symbol $L(a,b)$ stands for the (linear) operator on $E$ given by $L(a,b)(x)=\{ a, b, x\}$ ($x\in E$);
	\item The operator $L(a,a)$ is a hermitian operator with non-negative spectrum for each $a\in E$;
	\item $\|\{a,a,a\}\|=\|a\|^3$ for $a\in E$.
\end{enumerate}

Every C$^*$-algebra is a JB$^*$-triple with triple product $\{x,y,z\} =\frac12 (x y^* z +z y^* x).$ The same triple product is also valid to induce a structure of JB$^*$-triple on the space  $B(H,K)$ of all bounded linear operators between complex Hilbert spaces $H$ and $K$. For our purposes here, we note that every JB$^*$-algebra $\mathfrak{A}$ is a JB$^*$-triple for the triple product given by $\J xyz = (x\circ y^*) \circ z + (z\circ y^*)\circ x - (x\circ z)\circ y^*$ (see \cite[Theorem 4.1.45]{Cabrera}).\smallskip

Partial isometries in a C$^*$-algebra are precisely the fixed points of the triple product given above. An element $e$ in a JB$^*$-triple $E$ is said to be a \emph{tripotent} if $e= \J eee$. Any such element $e\in E$ induces a Peirce decomposition of $E$ in the form  $$E= E_{2} (e) \oplus E_{1} (e)\oplus E_0 (e),$$ where $E_i(e)$ corresponds to the eigenspace of $L(e, e)$ corresponding to the eigenvalue $\frac{i}{2}$ (see \cite[\S 4.2.2]{Cabrera} and \cite[\S 5.7]{Cabrera-Rodriguez-vol2}). The natural projection of $E$ onto $E_j (e)$, is known as the \emph{Peirce}-$j$ \emph{projection} and is denoted by $P_{i} (e)$. It is known that these projections are determined by the following expressions: $$P_2 (e) = L(e,e)(2 L(e,e) -Id_{E})=Q(e)^2,$$ $$ P_1 (e) = 4
L(e,e)(Id_{E}-L(e,e)) =2\left(L(e,e)-Q(e)^2\right),$$ $$ \ \hbox{ and } P_0 (e) =
(Id_{E}-L(e,e)) (Id_{E}-2 L(e,e)).$$ A tripotent $e\in E$ is called \emph{unitary} if $E= E_2(e)$.\smallskip

It is worthy remarking that unital JB$^*$-algebras are in one-to-one correspondence with JB$^*$-triples admitting a unitary element. More concretely, for each tripotent $e$ in a JB$^*$-triple, $E$, the Peirce-2 subspace $E_2 (e)$ is a unital JB$^*$-algebra with unit $e$, product $a\circ_{e} b := \{ a,e,b\}$ and involution $a^{*_e} := \J eae$. By a celebrated result by Kaup (see \cite[Proposition 5.5]{Kaup83}), the triple product on $E_2 (e)$ is uniquely determined by the expression
\begin{equation}\label{eq product Peirce2 as JB*-algebra} \{ a,b,c\} =(a \circ_{e} b^{*_e}) \circ_{e} c +(c \circ_{e} b^{*_e}) \circ_e a - (a \circ_e c) \circ b^{*_e},
\end{equation}
for every $a,b,c\in E_2 (e)$. Moreover, if $\mathfrak{A}$ is a unital JB$^*$-algebra, unitaries in $\mathfrak{A}$ for the structure of JB$^*$-algebra coincide with the unitary tripotents  for the structure of JB$^*$-triple (cf. \cite[Lemma 2.1$(d)$]{CuetoPeralta2023}). If $u$ is a unitary in $\mathfrak{A} = \mathfrak{A}_2 (u)$, we write $\mathfrak{A}(u) = (\mathfrak{A}, \circ_{u}, *_{u})$ for the unital JB$^*$-algebra defined above. Note that, as commented above, $\mathcal{U} (\mathfrak{A}) = \mathcal{U} ((\mathfrak{A}, \circ_{u}, *_{u}))$. We should note that there exists unital JB$^*$-algebras $\mathfrak{A}$ admitting a unitary element $u$ such that $T(\unit)\neq u$ for every surjective linear isometry $T$ on $\mathfrak{A}$ \cite[Example 5.7]{braun1978}.  \smallskip
 
\begin{lemma}\label{l piecewise +quadratic maps send unitaries to unitaries} Let $\Phi : \mathfrak{A}\to \mathfrak{B}$ be a quadratic mapping on operator commuting elements between two JBW$^*$-algebras. Then the following statements hold:\begin{enumerate}[$(a)$]
\item  $\Phi$ maps unitaries in $\mathfrak{A}$ to tripotents in $\mathfrak{B}$. In particular, $\Phi(\unit)$ is a tripotent in $\mathfrak{B},$ and $\Phi \left(\mathfrak{A}\right)\subseteq \mathfrak{B}_2 (\Phi(\unit))$;
\item $\Phi$ maps self-adjoint (resp., positive) elements in $\mathfrak{A}$ to self-adjoint (resp., positive) elements in the JBW$^*$-algebra $\mathfrak{B}_2 (\Phi(\unit))$, respectively. Moreover, $\Phi$ maps symmetries in $\mathfrak{A}$ to self-adjoint tripotents in $\mathfrak{B}_2 (\Phi(\unit))$;
\item If $a$ and $b$ are operator commuting positive elements in $\mathfrak{A}$, the elements $\Phi(a)$ and $\Phi(b)$ operator commute in the JBW$^*$-algebra $\mathfrak{B}_2 (\Phi(\unit))$ and $$ \Phi(a\circ b) = \Phi(a)\circ_{\Phi(\unit)} \Phi(b).$$ Consequently, $\Phi|_{\mathfrak{A}^{+}} : \mathfrak{A}^{+}\to (\mathfrak{B}_2 (\Phi(\unit)))^{+}$ is a piecewise Jordan homomorphism. 
\item If $\Phi|_{\mathfrak{A}_{sa}}$ is also additive on operator commuting elements, then $\Phi (0)=0$, $\Phi (-a) =- \Phi (a),$ for all $a\in \mathfrak{A}_{sa}$, the restriction of $\Phi$ to each associative subalgebra of $\mathfrak{A}_{sa}$ is Jordan-multiplicative with respect to the product of $\mathfrak{B}_2 (\Phi(\unit)),$ and all the elements in its image operator commute in the latter JBW$^*$-algebra.
\end{enumerate}
\end{lemma}

\begin{proof} $(a)$ Let us take $u\in \mathcal{U}(\mathfrak{A})$. Since $u$ and $u^*$ operator commute (see Lemma~\ref{lemma_op_comm_unit}), it follows from the hypotheses that $$\Phi (u) = \Phi \left( U_u (u^*) \right) = U_{\Phi(u)} \left(\Phi(u)^*\right) =\{\Phi(u), \Phi(u), \Phi(u)\},$$ that is, $\Phi (u)$ is a tripotent in $\mathfrak{B}$. In particular $\Phi (\unit)$ is a tripotent. To simplify the notation we shall write $w= \Phi(\unit)$.\smallskip 
	
Every element $a\in \mathfrak{A}$ operator commute with $\unit$, and hence $$\Phi (a) = \Phi \left( U_{\unit} (a) \right) = U_{\Phi(\unit)} \left(\Phi(a^*)^*\right) =\{w, \Phi(a^*), w\}\in  \mathfrak{B}_2 (w),$$ for all $a\in \mathfrak{A}$.\smallskip

$(b)$ If we take $a=a^*\in \mathfrak{A}$, we deduce, as above, that $$\Phi (a) = \Phi (a^*) = \Phi \left( U_{\unit} (a^*) \right) = \{w, \Phi(a), w\} = \Phi(a)^{*_w},$$ which shows that $\Phi (a)$ is self-adjoint in  $\mathfrak{B}_2 (w)$.  If $a$ is a positive element in $\mathfrak{A}$, we can write $a= b^2$ for some $b\in  \mathfrak{A}_{sa}$ such that $a,$ $b,$ and $\unit$ operator commute, and thus $$\Phi (a) = \Phi (b^2) = \Phi \left( U_{b} (\unit) \right) = \{\Phi(b), w, \Phi(b)\} = \Phi(b) \circ_{w} \Phi(b).$$ Observe that symmetries in $\mathfrak{A}$ are precisely the self-adjoint unitaries in $\mathfrak{A}$. So, the final conclusion in $(b)$ is a consequence of $(a)$ what have been already proved.\smallskip

$(c)$ Observe that for each $a\in \mathfrak{A}^{+}$ we can define, via functional calculus, $a^{\frac12}\in \mathfrak{A}^{+}$ belonging to the JB$^*$-subalgebra of $\mathfrak{A}$ generated by $a$ and satisfying $(a^{\frac12})^2 = a$. By $(b)$, $\Phi (a) =  \Phi(a^{\frac12}) \circ_{w} \Phi(a^{\frac12}),$ and hence $\Phi (a)$ belongs to the JB$^*$-subalgebra of $\mathfrak{B}_2 (w)$ generated by the positive element $\Phi(a^{\frac12})$ and $\Phi(a^{\frac12}) = \Phi(a)^{\frac12}$, where the latter is computed in  $\mathfrak{B}_2 (w)$. To avoid confusion, we shall write $U^{w}$ for the $U$ operator in the JBW$^*$-algebra $\mathfrak{B}_2 (w)$.\smallskip

Take now two operator commuting elements $a$ and $b$ in $\mathfrak{A}^+$. Observe that $a$ and $b^{\frac12}$ (respectively, $b$ and $a^{\frac12}$) operator commute in $\mathfrak{A}$ (cf. \cite[Theorem 3.12]{vandeWet2020}). By \cite[Proposition 1.1]{Hamhalter2025} the identity $ U_{a^{1/2}}(b) = U_{b^{1/2}}(a)$ holds. Therefore, by the previous comments and the hypotheses, we get: 
$$ \begin{aligned} U^{w}_{\Phi (a)^{1/2}}(\Phi(b)) &= U^{w}_{\Phi (a^{1/2})}(\Phi(b)) =\{\Phi (a^{1/2}),\Phi(b) ,\Phi (a^{1/2})\} = \Phi(\{a^{\frac12},b,a^{\frac12}\})\\
	& =\Phi (U_{a^{1/2}}(b))  = \Phi (U_{b^{1/2}}(a)) =\Phi (\{b^{\frac12}, a, b^{\frac12}\}) \\
	&= \{\Phi (b^{\frac12}), \Phi (a), \Phi (b^{\frac12})\}  =  U^{w}_{\Phi (b^{1/2})}(\Phi(a)) = U^{w}_{\Phi (b)^{1/2}}(\Phi(a)),
\end{aligned}
$$ where at the second and at penultimate equalities we applied \eqref{eq product Peirce2 as JB*-algebra}. A new application of \cite[Proposition 1.1]{Hamhalter2025} in the JBW$^*$-algebra $\mathfrak{B}_2 (w)$ proves that $\Phi(a)$ and $\Phi(b)$ operator commute in  $\mathfrak{B}_2 (w)$. Consequently,
$$\begin{aligned}\begin{aligned}
\Phi(a^2\circ b) &= \Phi(U_a(b)) = \Phi \{a,b,a\}= \{\Phi(a),\Phi(b), \Phi(a)\} = U^{w}_{\Phi(a)}(\Phi(b)) \\ &= (\Phi(a)\circ_{w} \Phi(a))\circ_w \Phi(b).		
	\end{aligned}
\end{aligned}$$ Having in mind that every positive element in $\mathfrak{A}$ admits a square root, together with the conclusion in the first paragraph of the proof of this statement, we finally get that 
$$ \Phi(a\circ b) = \Phi(a)\circ_{w} \Phi(b),$$ for every $a,b \in \mathfrak{A}_+$ such that $a,b$ operator commute, i.e. $\Phi|_{\mathfrak{A}^{+}} : \mathfrak{A}^{+}\to (\mathfrak{B}_2 (w))^{+}$ is a piecewise Jordan homomorphism.\smallskip

$(d)$ Since $0$ operator commute with itself we have $\Phi (0) = \Phi (0+0) = \Phi(0)+ \Phi(0)$, which proves that $\Phi(0)=0$. Every $a\in \mathfrak{A}_{sa}$ operator commute with its opposite, so $0=\Phi (0) = \Phi (a + (-a)) = \Phi (a) + \Phi (-a)$. \smallskip

Take two operator commuting elements $a,b\in \mathfrak{A}_{sa}$. The JB$^*$-subalgebra $\mathfrak{J}$ generated by $a, b$ and $\unit$ is a commutative unital C$^*$-subalgebra, and hence any two elements in $\mathfrak{J}$ operator commute in $\mathfrak{A}$ (cf. \cite[Theorem 3.12]{vandeWet2020}). Working in $\mathfrak{J}$ we can find pairwise operator commuting positive elements $a^+,a^-, b^+,b^-$ such that $a = a^+- a^-$ and $b=b^+ - b^-$. By applying $(c)$ we deduce that $\Phi(a^+),$ $\Phi(a^-),$ $\Phi(b^+),$ and $\Phi(b^-)$ pairwise operator commute in $\mathfrak{B}_2 (w)$. Therefore $\Phi (a) = \Phi(a^+)- \Phi(a^-)$ and $\Phi (b) = \Phi(b^+)- \Phi(b^-)$ operator commute. Finally,  
$$\begin{aligned}
\Phi (a) \circ_{w} \Phi (b) &=\Phi (a^+- a^-)\circ_w \Phi (b^+- b^-) \\ &= (\Phi (a^+) - \Phi (a^-))\circ_w (\Phi (b^+) - \Phi (b^-)) \\
&=\sum_{\sigma, \tau\in \{\pm\}} \sigma \ \tau\ \Phi (a^{\sigma}) \circ_{w} \Phi (b^{\tau}) = \Phi \left( \sum_{\sigma, \tau\in \{\pm\}} \sigma \ \tau\ a^{\sigma} \circ b^{\tau} \right)\\
&= \Phi \left(a \circ b\right).
\end{aligned}$$
\end{proof}

The arguments in the proof of Lemma~\ref{l piecewise +quadratic maps send unitaries to unitaries} are valid to deduce the following.

\begin{lemma}\label{r *-quadratic maps on self-adjoint elements} 
Let $\Phi : \mathfrak{A}_{sa}\to \mathfrak{B}_{sa}$ be a quadratic mapping on operator commuting elements, where $\mathfrak{A}$ and $\mathfrak{B}$ are  JBW$^*$-algebras. Then the following statements hold:\begin{enumerate}[$(a)$]
	\item  $\Phi$ maps symmetries in $\mathfrak{A}_{sa}$ to tripotents in $\mathfrak{B}_{sa}$ (i.e., elements which are written as the difference of two orthogonal projections). In particular, $\Phi(\unit)$ is a tripotent in $\mathfrak{B}_{sa},$ and $\Phi \left(\mathfrak{A}_{sa}\right)\subseteq \mathfrak{B}_2 (\Phi(\unit))_{sa} \cap \mathfrak{B}_{sa}$;
	\item $\Phi$ maps positive elements in $\mathfrak{A}_{sa}$ to positive elements in the JBW$^*$-algebra $\mathfrak{B}_2 (\Phi(\unit))$;
	\item If $a$ and $b$ are operator commuting positive elements in $\mathfrak{A}$, the elements $\Phi(a)$ and $\Phi(b)$ operator commute in the JBW$^*$-algebra $\mathfrak{B}_2 (\Phi(\unit))$ and $$ \Phi(a\circ b) = \Phi(a)\circ_{\Phi(\unit)} \Phi(b).$$ Consequently, $\Phi|_{\mathfrak{A}^{+}} : \mathfrak{A}^{+}\to (\mathfrak{B}_2 (\Phi(\unit)))^{+}$ is a piecewise Jordan homomorphism. 
	\item If $\Phi$ is also additive on operator commuting elements, then $\Phi (0)=0$, $\Phi (-a) =- \Phi (a),$ for all $a\in \mathfrak{A}_{sa}$,  the restriction of $\Phi$ to each associative subalgebra of $\mathfrak{A}_{sa}$ is Jordan-multiplicative with respect to the product of $(\mathfrak{B}_2 (\Phi(\unit)))_{sa},$ and all the elements in its image operator commute in the latter JBW-algebra.
\end{enumerate}	 
\end{lemma}

We can now state a generalization of the result by Friedman and Hakeda in \cite{FriHak88} for JBW-algebras. The result is new even in the case of von Neumann algebras.

\begin{theorem}\label{t FriedmanHakeda quadratic and additive on oc elements} Let $\Phi : \mathfrak{A}_{sa}\to \mathfrak{B}_{sa}$ be a quadratic mapping on operator commuting elements, where $\mathfrak{A}$ and $\mathfrak{B}$ are  JBW$^*$-algebras and $\mathfrak{A}$ admits no direct summands of type $I_2$. Suppose, additionally, that $\Phi $ is additive on operator commuting elements. Then $\Phi: \mathfrak{A}_{sa} \to \mathfrak{B}_2 (\Phi(\unit))_{sa}$ is a unital Jordan homomorphism. Furthermore, if $\Phi$ is bijective, we can conclude that $\Phi (\unit)$ is a central symmetry in $\mathfrak{B}$ and $\Phi$ is a Jordan isomorphism from $\mathfrak{A}_{sa}$ to $( \mathfrak{B}, \circ_{_{\Phi(\unit)}}, *_{_{\Phi(\unit)}})$. 
\end{theorem}

\begin{proof} As in the proof of the previous Lemma~\ref{l piecewise +quadratic maps send unitaries to unitaries}, we write $w = \Phi (\unit)$ which is a self-adjoint tripotent in $\mathfrak{B}$ (cf. Lemma~\ref{r *-quadratic maps on self-adjoint elements}). Since $\unit$ operator commute with itself, it follows from the hypotheses, and a simple induction argument, that $\Phi ( n \unit) = n \Phi (\unit) = n w$ for all natural $n$, and by Lemma~\ref{r *-quadratic maps on self-adjoint elements}$(d)$, $\Phi (q \unit ) = q \Phi (\unit) = q w,$ for all $q\in \mathbb{Q}$. Having in mind that $\Phi$ is Jordan multiplicative on operator commuting elements in $\mathfrak{A}_{sa}$ (it suffices to note that $\Phi (t \unit) \leq \Phi (s\unit)$ for all $t\leq s$ in $\mathbb{R}$), it is easy to check that $\Phi (t \unit) = t\Phi (\unit) = t w,$ for all $t\in \mathbb{R}$. \smallskip  
	
Let us take $a= a^*\in \mathfrak{A}$. Since $a$ operator commutes with $\unit$ and $\Phi$ is additive on operator commuting elements, we deduce that $$\Phi ( t^2 a) = \Phi \left( \{t \unit , a, t \unit\}  \right)= \{ \Phi \left( t \unit\right) ,  \Phi \left(a\right),  \Phi \left( t \unit\right)\} = \{ t w ,  \Phi \left(a\right),  t w\} = t^2 \Phi (a),$$ for all real $t$, and by Lemma~\ref{r *-quadratic maps on self-adjoint elements}$(d)$, $\Phi (t a ) = t \Phi (a) $ for all $t\in \mathbb{R}$, that is,  $\Phi$ is $1$-homogeneous and clearly additive on operator commuting elements by hypothesis. \smallskip
	
	Given $a=a^*$ with $\|a\|\leq 1$. We can write $a = a^+ - a^-$, where $0\leq a^+, a^-$ and $a^+ \circ a^- =0$. We also know that in this case $0\leq a^+ a^- \leq \unit$. Since $\unit$, $a^+$ and $a^-$ operator commute, a new application of Lemma~\ref{r *-quadratic maps on self-adjoint elements}$(d)$ and $(b)$ assures that $$0\leq \Phi (a^+)  = \Phi(\unit-(\unit-a^+)) = \Phi(\unit) - \Phi(\unit-a^+) \leq \Phi(\unit) =w.$$	 We can similarly prove that $0\leq \Phi (a^-) \leq w$, and thus $$\|\Phi(a^+)\|, \|\Phi (a^-)\|\leq \|w\| = 1.$$  Since $a^+$ and $a^-$ operator commute, the first hypothesis on $\Phi$ and the previous conclusion lead to 
	$\| \Phi (a) \| = \| \Phi (a^+) - \Phi (a^-)\| \leq \|\Phi (a^+)\| + \|\Phi (a^-)\| \leq 2$. We have thus proved that $\Phi|_{\mathfrak{A}_{sa}} : {\mathfrak{A}_{sa}}\to (\mathfrak{B}_2 (\Phi(\unit)))_{sa}$ is bounded on the closed unit ball of $\mathfrak{A}_{sa}$. We are in a position to apply Theorem~\ref{t corollary of the Mackey-Gleason} to the mapping $\Phi|_{\mathfrak{A}_{sa}} : {\mathfrak{A}_{sa}}\to (\mathfrak{B}_2 (\Phi(\unit)))_{sa}$ to deduce that it is linear and continuous.\smallskip
	
Since every $a\in \mathfrak{A}_{sa}$ operator commutes with itself,  Lemma~\ref{r *-quadratic maps on self-adjoint elements}$(d)$ also assures that $\Phi (a^2) = \Phi(a)\circ_{w}\Phi(a)$, and thus $\Phi$ is a unital Jordan homomorphism from $\mathfrak{A}_{sa}$ to $\mathfrak{B}_2 (\Phi(\unit))$.\smallskip

For the final claim, observe that if $\Phi$ is surjective we have $\mathfrak{B}_{sa} = \Phi\left( \mathfrak{A}_{sa} \right)\subseteq \mathfrak{B}_2 (\Phi(\unit))_{sa} \cap \mathfrak{B}_{sa}\subseteq \mathfrak{B}_{sa}.$  It follows that $\Phi (\unit)$ is a symmetry in $\mathfrak{B}$. It only remains to prove that $\Phi(\unit)$ is central. To see this, we write $\Phi(\unit) = p-q$, where $p$ and $q$ are two orthogonal projections in $\mathfrak{B}$ with $p+q = \unit$, and observe that, by the hypotheses on $\Phi$, for each $a\in \mathfrak{A}_{sa}$ we have $$\begin{aligned}
U_{p}\left(\Phi(a)\right) &+ U_{q}\left(\Phi(a)\right) + 2 \{p,\Phi(a), q\} = \Phi (a) = \Phi (\{\unit, a,\unit\}) \\ 
&= \{\Phi (\unit), \Phi (a), \Phi (\unit)\} = U_{p}\left(\Phi(a)\right) + U_{q}\left(\Phi(a)\right) - 2 \{p,\Phi(a), q\}.
\end{aligned} $$ The surjectivity of $\Phi$ implies that $\{p,\mathfrak{B}_{sa}, q\}=\{0\},$ and hence $ \mathfrak{B}_{sa}=U_p(\mathfrak{B})\oplus
U_q(\mathfrak{B}),$ which assures that $p$, $q,$ and $\Phi(\unit)$ are central elements.
\end{proof}

\begin{example}\label{couterexample spin} We shall give an example showing that the previous Theorems~\ref{t corollary of the Mackey-Gleason} and \ref{t FriedmanHakeda quadratic and additive on oc elements} fail for JBW$^*$-algebras of type $I_2$. The simplest case is a JBW$^*$-algebra factor of type $I_2,$ also known as a spin factor. Spin factors are very well studied and described (see for example, \cite{Harris1974,Harris1981,HervIsi1992,FriedmanBook,HOS}, \cite[\S 2.5.9]{Cabrera} and the discussion in \cite[\S 1.4]{EscolanoPeraltaVillena2025_OpCommut}). One of the representations reads as follows: a spin factor $V$ is obtained from a complex Hilbert space $(H,\langle \cdot | \cdot \rangle),$  with dim$(H)\geq 3$, together with a conjugation (i.e., a conjugate-linear isometry of period-$2$), $x\mapsto \overline{x}$, on $H$, a norm-one element $\unit = \overline{\unit}$, which is the unit for the Jordan product $$x \circ y = \langle x| \unit  \rangle y + \langle y | \unit  \rangle x -\langle x| \overline{y}  \rangle \unit \ \ (x,y\in V),$$ where the involution and norm on $V$ are given by $$x^* = 2 \langle  \unit | x \rangle \unit -  \overline{x}, \hbox{ and } \| x \|^2 =  \|x\|_2^2 + \left( \|x\|_2^4
	- | \langle x |\overline{x} \rangle |^2 \right)^{\frac12} \  (x\in V),$$ respectively.\smallskip
	
	It is known that $V_{sa} = \mathbb{R}\unit \oplus H^{-}$, where $H^{-}=\{h\in H : \langle \unit, h\rangle =0, \ \overline{h} = -h \},$ which is a real subspace of $H$ with dimension $\geq 2$. We can always find a bijective mapping $F: H^{-}\to H^{-},$ which is $1$-homogeneous but not additive. The mapping $\Phi: V_{sa} \to V_{sa}$ defined by $\Phi (\lambda \unit + h ) = \lambda \unit + F(h)$ ($\lambda \in \mathbb{R}$, $h\in H^{-}$), is a well-defined bijection, which is clearly $1$-homogeneous but not additive.\smallskip  	
	
	We shall next show that $\Phi$ is additive on operator commuting elements in $V_{sa}$. To see this, recall that elements $a,b\in V$ operator commute if, and only if, $b$ is a linear combination of $a$ and $\unit$ or $a$ is a linear combination of $b$ and $\unit$ (cf. \cite[Remark 1.5]{EscolanoPeraltaVillena2025_OpCommut}). So, if $a$ and $b$ operator commute in $V_{sa}$, we can reduce to the case in which $a= \alpha \unit + h$ and $b = t \unit + s a = (t+ s\alpha)\unit + s h$, where $s,t,\alpha\in \mathbb{R}$, $h\in H^{-}$. It follows from the definition of $\Phi$ that $$\Phi ( a + b) =  (t+ s\alpha +\alpha )\unit + F((s+1) h) = (t+ s\alpha)\unit +(s+1)  F( h) = \Phi (a) + \Phi (b).$$ In fact, $\Phi$ is linear on the linear span of $\unit$ and $h$. \smallskip 
	
	We can further assume that the mapping $F: H^{-}\to H^{-}$ satisfies $\|F(k) \| = \|k\|$ for all $k\in H^{-}$. In such a case, for $a$ and $b$ as above we have $$\begin{aligned}
		U_{a} (b) &= \{a,b,a\} = \{\alpha \unit + h, (t+ s\alpha)\unit + s h ,\alpha \unit + h\} \\
		&= \left( \alpha^2 t + s \alpha^3 + (2\alpha s + t) \|h\|^2 \right)\unit + \left( 2 \alpha  t+ 3 s \alpha^2  \right) h.
	\end{aligned}$$ Having in mind that $\Phi$ is linear on the linear span of $\unit$ and $h$, with $\|\Phi(h)\| =\|h\|$, we deduce that $$  \begin{aligned} \Phi\left(
		U_{a} (b)\right) &= \left( \alpha^2 t + s \alpha^3 + (2\alpha s + t) \|h\|^2 \right)\unit + \left( 2 \alpha  t+ 3 s \alpha^2  \right) F(h)\\
		&= \left( \alpha^2 t + s \alpha^3 + (2\alpha s + t) \|F(h)\|^2 \right)\unit + \left( 2 \alpha  t+ 3 s \alpha^2  \right) F(h) \\
		&= U_{\alpha \unit + F(h)} \left((t+s \alpha) \unit + s F(h)\right) = U_{\Phi(a)} (\Phi(b)).
	\end{aligned}$$ Therefore, $\Phi$ is additive and a quadratic mapping on operator commuting elements in $V_{sa}$. 
\end{example}

It is worth to conclude this section by listing some additional properties of quadratic maps on operator commuting elements  between JBW$^*$-algebras (compare Lemma~\ref{l piecewise +quadratic maps send unitaries to unitaries}). If $w$ is a tripotent in a JB$^*$-algebra $\mathfrak{A},$ for each $a\in \mathfrak{A}_2 (w)$ we shall write $U^{w}_a$ for the $U$-operator on $\mathfrak{A}_2 (w)$ associated with an element $a\in \mathfrak{A}_2 (w)$. 

\begin{lemma}\label{l piecewise +quadratic maps send unitaries to unitaries surjective and bijective} Let $\Phi : \mathfrak{A}\to \mathfrak{B}$ be a quadratic mapping on operator commuting elements between two JBW$^*$-algebras such that $\mathfrak{A}$ admits no direct summands of type $I_2$. Then the following statements hold:
\begin{enumerate}[$(a)$]
\item If $\Phi({\mathfrak{A}_{sa}}) = (\mathfrak{B}_2 (\Phi(\unit)))_{sa}$ and $\Phi|_{\mathfrak{A}_{sa}}$ is additive on operator commuting elements, then there exists a central projection $z\in  \mathfrak{B}_2 (\Phi(\unit))$ satisfying $$\Phi (i \unit) = i (z- (\Phi(\unit)-z)).$$
\item If $\Phi$ satisfies the hypotheses in $(a)$ and $\Phi|_{\mathfrak{A}_{sa}} : {\mathfrak{A}_{sa}}\to (\mathfrak{B}_2 (\Phi(\unit)))_{sa}$ is bijective, then there exists a central projection $r\in \mathfrak{A}$ such that $\Phi (r) = z$, $\Phi (\unit-r) = \Phi (\unit)-z$, $$\Phi (i a) = i \Phi (a)\ (\forall a\in U_{r} (\mathfrak{A}_{sa})), \ \Phi (i b) = -i \Phi (b) \ (\forall b\in U_{\unit-r} (\mathfrak{A}_{sa})),$$ $\Phi \left( r\circ \mathfrak{A}_{sa} \right)\subseteq z\circ_w  \mathfrak{B}_2 (\Phi(\unit))_{sa},$ and $\Phi \left( (\unit-r)\circ \mathfrak{A}_{sa} \right)\subseteq ({w-z}) \circ_{w} \mathfrak{B}_2 (\Phi(\unit))_{sa}.$
\item If $\Phi$ satisfies the hypotheses in $(a)$ and $(b)$, and $\Phi (a+i b) = \Phi (a) + \Phi (ib)$ for all $a,b\in \mathfrak{A}_{sa}$, then $\Phi: \mathfrak{A}\to \mathfrak{B}_2 (\Phi(\unit))$ is a real linear Jordan $^*$-isomorphism, actually $\Phi$ is the direct sum of linear Jordan $^*$-isomorphism and a conjugate-linear Jordan $^*$-isomorphism between appropriate weak$^*$-closed ideals of $\mathfrak{A}$ and $\mathfrak{B}_2 (\Phi(\unit))$. 
\end{enumerate}
\end{lemma} 

\begin{proof} $(a)$ It follows from Lemma~\ref{l piecewise +quadratic maps send unitaries to unitaries}$(a)$ that $\Phi (i \unit)$ is a tripotent in $\mathfrak{B}_2 (w)$, where, as in the previous lemma, we write $w$ for the tripotent $\Phi (\unit)\in \mathfrak{B}$. The $U$ operators corresponding the JBW$^*$-algebra $\mathfrak{B}_2 (w)$ will be denoted by $U^{w}$. Since $i\unit$ and $\unit$ operator commute, we have $$\Phi (-i \unit) = \Phi (\{\unit, i \unit, \unit\}) = \{w,\Phi(i \unit), w\} = \Phi(i \unit)^{*_w}.$$ We can therefore write $\Phi (i \unit) = i (z-q)$, where $z$ and $q$ are two orthogonal projections in $\mathfrak{B}_2 (w)$. Furthermore, $$\begin{aligned}
		-\Phi (\unit)= \Phi (-\unit) &= \Phi (\{i\unit,\unit, i \unit\}) = \{\Phi(i \unit), w, \Phi(i \unit)\} \\ &= \Phi(i \unit)\circ_{w} \Phi(i \unit) = - (z+q),
	\end{aligned}$$ which assures that $q = w-z = \Phi(\unit)-z$. It follows from the assumptions that for each $b\in (\mathfrak{B}_2 (w))_{sa}$ there exists $a\in \mathfrak{A}_{sa}$ such that $\Phi (a) =b$. Since $a$ and $i\unit$ operator commute, having in mind  Lemma~\ref{l piecewise +quadratic maps send unitaries to unitaries}$(d)$, we have $$\begin{aligned}
		- U^{w}_z (b) - U^{w}_q (b)- 2 \{z, b, q\} &= - b=-\Phi (a) = \Phi \left(\{i \unit, a, i \unit\} \right)\\ &= \{i (z-q), \Phi (a), i (z-q)\} \\ &= - U^{w}_z (b) - U^{w}_q (b)+ 2 \{z, b, q\},
	\end{aligned} $$ and so $\{z, b, q\} = 0$, for all $b \in (\mathfrak{B}_2 (w))_{sa}$, and thus $z$ is a central projection in $\mathfrak{B}_2 (w)$.\smallskip
	
$(b)$ Assume now that $\Phi|_{\mathfrak{A}_{sa}} : {\mathfrak{A}_{sa}}\to (\mathfrak{B}_2 (w))_{sa}$ is a bijection. So, we can pick $r\in \mathfrak{A}_{sa}$ satisfying $\Phi (r) = z$. By hypotheses, $\Phi (r^2) = \Phi \left(\{r, \unit, r\}  \right) = \{z,w,z\} = z = \Phi (r),$ which proves that $r^2 = r$ is a projection. We deduce from  Lemma~\ref{l piecewise +quadratic maps send unitaries to unitaries}$(d)$ that $\Phi (\unit -r ) = \Phi (\unit)-\Phi(r) = w-z$.\smallskip
	
	 To see that $r$ is central, we recall that a projection $p$ in a JBW$^*$-algebra $\mathfrak{M}$ is central if, and only if, $p\circ a = U_p (a)$ for all $a\in \mathfrak{M}_{sa}$ (cf. \cite[Lemma 2.5.5]{HOS}). Note that $z$ is central in $\mathfrak{B}_2 (w)$, which can be applied in the identity  $$\begin{aligned}
		\Phi \left( U_r (a) \right) &= \{\Phi(r),\Phi(a),\Phi(r)\} =U^{w}_{z} (\Phi(a)) = z\circ_w \Phi(a) \\
		&= \Phi (r) \circ_w \Phi(a) = \hbox{(by  Lemma~\ref{l piecewise +quadratic maps send unitaries to unitaries}$(d)$)} = \Phi (r\circ a),
	\end{aligned}$$ where $a$ is an arbitrary element in $\mathfrak{A}_{sa}$. The injectivity of $\Phi$ assures that $ U_r (a) = r\circ a$ for all $a\in \mathfrak{A}_{sa}$, and thus $r$ is central. We have also shown that $\Phi \left( U_r (a) \right)\in U^w_{z} \left( \mathfrak{B}_2 (\Phi(\unit)) \right)$.\smallskip
	
	It follows from the conclusions above that $$\Phi (-i \unit) = \Phi(\{\unit, i \unit, \unit\}) = \Phi (i\unit)^{*_w} = -\Phi (i \unit).$$  
	
	To see the final statement, let us take $c\in U_{r} (\mathfrak{A})$ and observe that since $c$ and $i \unit$ operator commute it follows that  
	$$\begin{aligned}
		\Phi (i c^2) &= \Phi(\{c, - i \unit, c  \})= \{\Phi (c), \Phi(- i \unit),\Phi (c) \} = \{\Phi (c), -i (z-(w-z)),\Phi (c) \}  \\ 
		&= i \{\Phi (c), \Phi (r) ,\Phi (c) \} -i \{\Phi (c), \Phi (\unit-r),\Phi (c) \}  \\
		& = i \Phi ( \{c, r ,c\}) -i \Phi ( \{c, \unit-r, c \} = i  \Phi (c^2).
	\end{aligned} $$ The arbitrariness of $c\in U_{r} (\mathfrak{A})$ allows us to deduce that $\Phi (i a) = i \Phi (a),$ for all  $a\in U_{r} (\mathfrak{A}_{sa})$. We can similarly get $\Phi (i a) = -i \Phi (a)$ for all  $a\in U_{\unit-r} (\mathfrak{A}_{sa}).$\smallskip

$(c)$ It follows from Theorem~\ref{t FriedmanHakeda quadratic and additive on oc elements} that the mapping  $\Phi|_{\mathfrak{A}_{sa}} : {\mathfrak{A}_{sa}}\to (\mathfrak{B}_2 (\Phi(\unit)))_{sa}$ is a Jordan isomorphism. The additional assumptions together with $(a)$ and $(b)$ guarantee that the mapping $\Phi_1 =\Phi|_{r\circ \mathfrak{A}} : {r\circ \mathfrak{A}}\to {z\circ_w  \mathfrak{B}_2 (\Phi(\unit))}$ is a well-defined complex linear Jordan $^*$-isomorphism, and similarly, $\Phi_2= \Phi|_{(\unit-r)\circ \mathfrak{A}} : {(\unit-r)\circ \mathfrak{A}}\to {(w-z)\circ_w  \mathfrak{B}_2 (\Phi(\unit))}$ is a well-defined conjugate-linear Jordan $^*$-isomorphism. Recall that $w= \Phi (\unit)$. Each $x\in \mathfrak{A}$ operator commutes with $r$ and $\unit-r$, and thus 
$$U_{z} (\Phi (x))^{*_w}  =\{\Phi(r), \Phi (x), \Phi (r)\} = \Phi\{r,x,r\} = \Phi \left(r\circ x^*\right) =\Phi_1 (r\circ x)^{*_w},$$ and $ U_{w-z} (\Phi (x)) = \Phi_2 ((\unit-r)\circ x)$. Finally, $$\Phi (x) = U_{z} (\Phi (x)) + U_{w-z} (\Phi (x)) = \Phi_1 (r\circ x) +  \Phi_2 ((\unit-r)\circ x),$$ witnessing that $\Phi$ is the direct sum of linear Jordan $^*$-isomorphism and a conjugate-linear Jordan $^*$-isomorphism.
\end{proof}

We conclude with the following corollary.

\begin{corollary}\label{cor invariant for JBWstar algebras} Let $\mathfrak{A}$ and $\mathfrak{B}$ be JBW$^*$-algebras such that $\mathfrak{A}$ admits no direct summands of type $I_2$. Then the following statements are equivalent: \begin{enumerate}[$(a)$]
\item $\mathfrak{A}$ and $\mathfrak{B}$ are Jordan $^*$-isomorphic. 
\item There exists a bijection $\Phi : \mathfrak{A}_{sa}\to \mathfrak{B}_{sa}$ which is additive and a quadratic mapping on operator commuting elements.
\end{enumerate}
\end{corollary}

\medskip\medskip

\textbf{Acknowledgements}\medskip

G.M. Escolano, A.M. Peralta and A.R. Villena were partially supported by
MCIN/AEI/10.13039/501100011033 and by ``ERDF A way of making Europe'' grant PID2021-122126NB-C31, by the IMAG--Mar{\'i}a de Maeztu grant CEX2020-001105-M/AEI/10.13039/ 501100 011033, and by Junta de Andalucía­ grants no. FQM185, and FQM375.
First author supported by grant FPU21/00617 at University of Granada founded by Ministerio de Universidades (Spain).

\smallskip\smallskip

\noindent\textbf{Data Availability} Statement Data sharing is not applicable to this article as no datasets were generated or analyzed during the preparation of the paper.\smallskip\smallskip

\noindent\textbf{Declarations}
\smallskip\smallskip

\noindent\textbf{Conflict of interest} The authors declare that he has no conflict of interest.


\begin{thebibliography}{0}
	
\bibitem{AlfsenShultz2003} E. M. Alfsen and F.W. Shultz, \emph{Geometry of state spaces of operator algebras}, Mathematics: Theory \& Applications, Birkh\"{a}user Boston, Inc., Boston, MA, 2003.

\bibitem{braun1978} R. Braun, W. Kaup, H. Upmeier, A holomorphic characterization of {J}ordan {C}$^*$-algebras, \emph{Mathematische Zeitschrift} \textbf{161} (1978), 3, 277--290.

\bibitem{Cabrera} M. Cabrera, A. Rodr\'{i}guez~Palacios,  \emph{Nonassociative Normed Algebras. Vol.1. The Vidav-Palmer and Gelfand-Naimark theorems}. Encyclopedia Math. Appl., 154
Cambridge University Press, Cambridge, 2014.

\bibitem{Cabrera-Rodriguez-vol2} M. Cabrera, A. Rodr\'{i}guez~Palacios, \emph{Non-associative normed algebras. {V}ol. 2. Representation theory and the Zel'manov approach.}, vol.~167 of {\em Encyclopedia Math. Appl.} 167. Cambridge University Press, Cambridge, 2018.

\bibitem{CuetoPeralta2023} M. Cueto-Avellaneda, A. M. Peralta, \emph{Can one identify two unital JB*-algebras by the metric spaces determined by their sets of unitaries?}, \emph{Linear Multilinear Algebra} \textbf{70} (2022), no. 22, 7702--7727.

\bibitem{Dye} H.A. Dye: On the geometry of projections in certain operator algebras, \emph{Ann. Math.} \textbf{61} (1955), (2) 73--89. 

\bibitem{Edwards80} C.M. Edwards, On Jordan W$^*$-algebras, \emph{Bull. Sci. Math.} \textbf{104} (1980), 393--403.

\bibitem{EscolanoPeraltaVillena2025_OpCommut} G. M. Escolano, A. M. Peralta, A. R. Villena, \emph{Preservers of Operator Commutativity}, J. Math. Anal. Appl. \textbf{552} (2025), no. 2, Paper No. 129796, 64 pp. 

\bibitem{EscolanoPeraltaVillena2025_mes} G. M. Escolano, A. M. Peralta, A. R. Villena, \emph{The Mackey--Gleason--Bunce--Wright problem for vector-valued measures on projections in a JBW$^*$-algebra}, arXiv:2509.03213 

\bibitem{FriedmanBook} Y. Friedman, \emph{Physical applications of homogeneous balls}. With the assistance of Tzvi Scarr. Progress in Mathematical Physics 40. Boston, MA: Birkh\"{a}user, 2005.

\bibitem{FriHak88} Y. Friedman, J. Hakeda, Additivity of quadratic maps, \emph{Publ. Res. Inst. Math. Sci.} \textbf{24} (1988), No. 5, 707--722.

\bibitem{GarPe2021} J.J. Garcés, A.M. Peralta, One-parameter semigroups of orthogonality preservers on JB$^*$-algebras, \emph{Adv Oper Theory} \textbf{6} (2021): Paper No. 43.

\bibitem{Hak86} J. Hakeda, J., Additivity of $^*$-semigroup isomorphisms among $^*$-algebra, \emph{Bull. London Math. Soc.} \textbf{18} (1986), 51--56.

\bibitem{HakSa86} J. Hakeda, K. Saito, Additivity of Jordan $^*$-maps between operator algebras, \emph{Math. Soc. Japan} \textbf{38} (1986), 403--408.

\bibitem{HOS} H. Hanche-Olsen, E. St{\o}rmer, \emph{Jordan Operator Algebras}, Pitman, London, 1984.


\bibitem{Hamhalter2018} J. Hamhalter , \emph{Piecewise $^*$-homomorphisms and Jordan maps on C$^*$-algebras and factor von Neumann algebras}, \emph{J. Math. Anal. Appl.} \textbf{462} (2018), no. 1, 1014--1031.


\bibitem{Hamhalter2023} J. Hamhalter, \emph{Maps preserving products of commuting elements in von Neumann algebras}, \emph{J. Math. Anal. Appl.} \textbf{523} (2023), no.2, Paper No. 127044.


\bibitem{Hamhalter2025} J. Hamhalter , \emph{Morphisms of positive cones of JBW-algebras and von Neumann algebras that preserve products of operator commuting elements}, preprint 2025.

\bibitem{Harris1974} L.A. Harris, Bounded symmetric homogeneous domains in infinite dimensional spaces. In \emph{Proc. infinite dim. Holomorphy, Lexington 1973}, Lect. Notes Math. \textbf{364} (1974), 13-40.

\bibitem{Harris1981} L.A. Harris, A generalizarion of C$^*$-algebras, \emph{Proc. London Math. Soc.} \textbf{42} (1981), 331--361.

\bibitem{HervIsi1992} F.J. Herv{\'e}s, J.M. Isidro, Isometries and automorphisms of the spaces of spinors, \emph{Rev. Mat. Univ. Complutense Madrid} \textbf{5} (1992), No. 2-3, 193--200.

\bibitem{HeunenReyes2014} C. Heunen, M.L. Reyes,  \emph{Active lattices determine AW$^*$-algebras}, \emph{J. Math. Anal. Appl.} \textbf{416} (2014), no. 1, 289--313.

\bibitem{HewRossBookVol1} E. Hewitt, K.A. Ross, \emph{Abstract harmonic analysis. Volume I: Structure of topological groups, integration theory, group representations.} 2nd ed. Grundlehren der Mathematischen Wissenschaften. 115. Berlin: Springer- Verlag, 1994.

\bibitem{KadRingVolI} R.V. Kadison, J.R. Ringrose, Fundamentals of the theory of operator algebras. Vol. I. New York
(NY): Academic Press Inc. 1983.

\bibitem{Kaup83} W. Kaup, A Riemann mapping theorem for bounded symmetric domains in complex Banach spaces, \emph{Math. Z.} \textbf{183} (1983), 503--529.

\bibitem{KoSpeck1967} S. Kochen, E.P. Specker, The problem of hidden variables in quantum mechanics, \emph{J. Math. Mech.} \textbf{17} (1967), 59-87.

\bibitem{Molnar2015} L. Moln\'ar,  \emph{General Mazur-Ulam type theorems and some applications}. In \emph{Operator semigroups meet complex analysis, harmonic analysis and mathematical physics} (Arendt, Wolfgang ed. et al.). Proceedings of the conference, Herrnhut, Germany, June 3?7, 2013. Operator Theory: Advances and Applications 250, 311-342 (2015). 

\bibitem{Ozawa2011} M. Ozawa, Universal uncertainty principle, simultaneous measurability, and weak values. In Ralph, Timothy (ed.) et al., Quantum communication, measurement and computing (QCMC). The tenth international conference, Brisbane, Australia, July 19?23, 2010. Melville, NY: American Institute of Physics (AIP). AIP Conference Proceedings 1363, 53-62 (2011).

\bibitem{Topping65} D.M. Topping,  \emph{Jordan algebras of self-adjoint operators}, Mem. Am. Math. Soc, \textbf{53}, 1965.

\bibitem{vandeWet2020} J. van de Wetering, \emph{Commutativity in Jordan operator algebras}, J. Pure Appl. Algebra \textbf{224} (2020), no. 11, 106407.

\bibitem{Wright1977} J.D.M. Wright, Jordan C$^*$ algebras, \emph{Michigan Math J.} \textbf{24} (1977), 291--302.


\end{thebibliography}
\end{document}